\documentclass[10pt]{article}

\usepackage[T1]{fontenc}
\usepackage[utf8]{inputenx}
\usepackage[english]{babel}
\usepackage{microtype}
\usepackage{framed}
\usepackage{listings}
\usepackage{vmargin}
\usepackage{setspace}
\usepackage{mathrsfs, mathenv}
\usepackage{amsmath, amsthm, amssymb, amsfonts, amscd}
\usepackage{graphicx}
\usepackage{epstopdf}
\usepackage[svgnames]{xcolor}
\usepackage[numbers]{natbib}
\usepackage{hyperref}
\usepackage[capitalise]{cleveref}
\usepackage{xparse}

\hypersetup{citecolor=blue, colorlinks=true, linkcolor=black}

\newcommand{\operator}[0]{\mathcal}
\newcommand{\grad}[0]{\nabla}
\newcommand{\vect}[1]{\boldsymbol{\mathbf{#1}}}
\newcommand{\hermite}[1]{#1}
\newcommand{\hermitef}[1]{#1}
\newcommand{\expect}[0]{\mathbf{E}}
\newcommand{\proj}[2]{\pi\left(#1,#2\right)}

\newcommand{\seq}[4]{{\left\{#1_{#2}\right\}}_{#2=#3}^{#4}}

\DeclareMathOperator{\trace}{tr}

\DeclareMathOperator{\Span}{span}

\newcommand{\nat}[0]{\mathbf{N}}
\newcommand{\real}[0]{\mathbf{R}}
\newcommand{\poly}[1]{\mathbf{P}_{#1}}
\newcommand{\cont}[2]{C^{#1} (#2)}

\newcommand{\smooth}[1]{C^\infty(#1)}
\newcommand{\test}[1]{C^\infty_c (#1)}
\newcommand{\schwartz}[1]{S (#1)}

\newcommand{\abs}[1]{\left|#1\right|}

\DeclareDocumentCommand\sobolev{m m o} {H^{#1}\left(#2 \IfNoValueF{#3}{,#3}\right)}
\DeclareDocumentCommand\lp{m m o} {L^{#1}\left(#2 \IfNoValueF{#3}{,#3}\right)}
\DeclareDocumentCommand\gaussian{O{0} O{I}} {G_{(#1, #2)}}
\DeclareDocumentCommand\norm{m o o} {\|#1\|\IfNoValueF{#2}{_{#2 \IfNoValueF{#3}{,#3}}}}
\DeclareDocumentCommand\seminorm{m o o} {\left|#1\right|\IfNoValueF{#2}{_{#2 \IfNoValueF{#3}{,#3}}}}
\DeclareDocumentCommand\ip{m m o o} {\langle{#1,#2}\rangle\IfNoValueF{#3}{_{#3 \IfNoValueF{#4}{,#4}}}}
\DeclareDocumentCommand\dup{m m o} {\left\langle{#1,#2}\right\rangle\IfNoValueF{#3}{_{#3', #3}}}

\newcommand{\TheTitle}{Spectral methods for multiscale stochastic differential equations}

\newcommand{\email}[1]{\href{#1}{#1}}

\title{{\TheTitle}}

\author{
  A. Abdulle\thanks{Mathematics Section, École Polytechnique Fédérale de Lausanne
    (\email{assyr.abdulle@epfl.ch}).}
  \and
  G.A. Pavliotis\thanks{Department of Mathematics, Imperial College London
   (\email{g.pavliotis@imperial.ac.uk}).}
  \and
  U. Vaes\thanks{Department of Mathematics, Imperial College London
    (\email{u.vaes13@imperial.ac.uk}).}
}

\theoremstyle{plain}

\newtheorem{theorem}{Theorem}[section]
\newtheorem{corollary}[theorem]{Corollary}
\newtheorem{lemma}[theorem]{Lemma}
\newtheorem{proposition}[theorem]{Proposition}

\theoremstyle{definition}
\newtheorem{definition}[theorem]{Definition}
\theoremstyle{remark}
\newtheorem{remark}[theorem]{Remark}
\newtheorem{assumption}{Assumption}[section]
\crefname{assumption}{Assumption}{Assumptions}
\crefname{section}{Section}{Sections}

\begin{document}

\maketitle
\begin{abstract}
    This paper presents a new method for the solution of multiscale stochastic differential equations at the diffusive time scale.
    In contrast to averaging-based methods, e.g., the heterogeneous multiscale method (HMM) or the equation-free method, which rely on Monte Carlo simulations,
    in this paper we introduce a new numerical methodology that is based on a spectral method.
    In particular, we use an expansion in Hermite functions to approximate the solution of an appropriate Poisson equation,
    which is used in order to calculate the coefficients of the homogenized equation.
    Spectral convergence  is proved under suitable assumptions.
    Numerical experiments corroborate the theory and illustrate the performance of the method.
    A comparison with the HMM and an application to singularly perturbed stochastic PDEs are also presented.
\end{abstract}

\vspace{0.5cm}

\vspace{0.5cm}
{\bf AMS:} 
  65N35, 
  65C30, 
  60H10 
  60H15 

\setstretch{1.15}
\section{Introduction}
\label{sec:introduction}

Multiscale stochastic systems arise frequently in applications.
Examples include atmosphere/ocean science~\cite{MTV01} and materials science~\cite{weinan2011}.
For systems with a clear scale separation it is possible, in principle,
to obtain a closed---averaged or homogenized---equation for the slow variables~\cite{pavliotis2008multiscale}.
The calculation of the drift and diffusion coefficients that appear in this effective (coarse-grained) equation
requires appropriate averaging over the fast scales.
Several numerical methods for multiscale stochastic systems that are based on scale separation and
on the existence of a coarse-grained equation for the slow variables have been proposed in the literature.
Examples include the heterogeneous multiscale method (HMM)~\cite{vanden2003fast,weinan2005analysis,abdulle2012heterogeneous}
and the equation-free approach~\cite{kevrekidis2003equation}.
These techniques are based on evolving the coarse-grained dynamics,
while calculating the drift and diffusion coefficients ``on-the-fly'' using short simulation bursts of the fast dynamics.

A prototype fast/slow system of stochastic differential equations (SDEs) for which the aforementioned techniques can be applied is \footnote{
    In this paper we will consider the fast/slow dynamics at the diffusive time scale, or,
    using the terminology of~\cite{pavliotis2008multiscale}, the homogenization problem.
}
\begin{subequations}
    \label{eq:prototype_general_fast-slow_system}
    \begin{equation}\label{e:slow}
        d{X}_t^{\varepsilon} = \frac{1}{\varepsilon} \vect f({X}_t^{\varepsilon},{Y}_t^{\varepsilon}) \, dt + \sqrt 2 \, \vect \sigma_x \,d{W}_{xt},
    \end{equation}
    \begin{equation}\label{e:fast}
        d{Y}_t^{\varepsilon} = \frac{1}{\varepsilon^2} \vect h({X}_t^{\varepsilon},{Y}_t^{\varepsilon}) \,dt + \frac{\sqrt 2}{\varepsilon}\vect \sigma_y\,d{W}_{yt}.
    \end{equation}
\end{subequations}
where
$X_t^\varepsilon \in \real^m$, $Y_t^\varepsilon \in \real^n$,
$\varepsilon \ll 1$ is the parameter measuring scale separation,
$\vect \sigma_x \in \real^{m \times d_1}$, $\vect \sigma_y \in \real^{n \times d_2}$ are constant matrices,
and $W_x$, $W_y$ are independent $d_1$ and $d_2$-dimensional Brownian motions, respectively.\footnote{
    It is straightforward to consider problems where the Brownian motions driving the fast and slow processes are correlated.
    This scenario might be relevant in applications to mathematical finance. See e.g. \cite{MR2738765}.
}
For fast-slow systems of this form, a direct numerical approximation of the full dynamics would be prohibitively expensive,
because resolving the fine scales would require a time step $\delta t$ that scales as $\mathcal O(\varepsilon^2)$.
Under appropriate assumptions on the coefficients and on the ergodic properties of the fast process $Y_t^{\varepsilon}$,
it is well known that the slow process converges, in the limit as $\varepsilon$ tends to $0$,
to a homogenized equation that is independent of the fast process and of $\varepsilon$~\cite[Ch. 11]{pavliotis2008multiscale}:
\begin{equation}
    \label{eq:simplified_equation_for_general_prototype}
    dX_t = \vect F(X_t) \, dt + \vect A(X_t) \, dW_t.
\end{equation}
The drift and diffusion coefficients in~\eqref{eq:simplified_equation_for_general_prototype} can be calculated by
solving a Poisson equation involving the generator of the fast process,\footnote{
    We are assuming that the centering condition is satisfied, see \cref{eq:assumption_f_bounded_polynomial} below.
}
\begin{equation}
    \label{eq:poisson_equation_general_prototype}
     - \mathcal L_y \vect \phi = \vect f,
\end{equation}
where $\mathcal L_y = \vect h(x,y) \cdot \nabla_y + \vect \sigma_y^2 \Delta_y$, together with appropriate boundary conditions,
and calculating averages with respect to the invariant measure $\mu_x(dy)$ of $Y_t^{\varepsilon}$:
\begin{subequations}
    \label{eq:effective_drift_and_diffusion}
    \begin{align}
        \label{eq:effective_drift}
        \vect F(x) &= \int_{\real^n} \grad_x{\vect \phi(x,y)} \,\vect f(x,y) \,\mu_x(dy),\\
        \label{eq:effective_diffusion}
        \vect A(x)\vect A(x)^T &= \int_{\real^n} \left[\vect f(x,y) \otimes \vect \phi(x,y) + \vect \phi(x,y) \otimes \vect f(x,y)\right] \, \mu_x(dy).
    \end{align}
\end{subequations}
Once the drift and diffusion coefficients have been calculated,
then it becomes computationally advantageous to solve the homogenized equations,
in particular since we are usually interested in the evolution of observables of the slow process alone.
The main computational task, thus, is to calculate
the drift and diffusion coefficients that appear in the homogenized equation~\eqref{eq:simplified_equation_for_general_prototype}.
When the state space of the fast process is high dimensional,
the numerical solution of the Poisson equation
and calculation of the integrals in~\eqref{eq:poisson_equation_general_prototype} using deterministic methods
become prohibitively expensive and Monte Carlo-based approaches have to be employed.
In recent years different methodologies have been proposed
for the numerical solution of the fast-slow system~\eqref{eq:prototype_general_fast-slow_system}
that are based on the strategy outlined above,
for example the Heterogeneous Multiscale Method (HMM)~\cite{vanden2003fast,weinan2005analysis,abdulle2012heterogeneous}
and the equation-free approach~\cite{kevrekidis2003equation}.
In particular, the PDE-based formulas~\eqref{eq:effective_drift_and_diffusion} are replaced by
Green-Kubo type formulas~\cite[Sec. 1]{weinan2005analysis} that involve time averages and numerically calculated autocorrelation functions.
The equivalence between the homogenization and the Green-Kubo formalism
has been shown for a quite general class of fast/slow systems of SDEs \cite{MR2740040}. See also \cite{MR3463433,MR3509213}.
While offering several advantages, time and ensemble averages, on which these methods are based,
imply that accurate solutions are computationally very expensive to obtain.
Based on the analysis of~\cite{weinan2005analysis}, one deduces that
the computational cost needed to obtain an error of order $2^{-p}$ scales as $\mathcal O(2^{p(2 + 1/l)})$,
where $l$ is the weak order of accuracy of the micro-solver used.

When the dimension of the state space of the fast process is relatively low,
numerical approaches that are based on the accurate and efficient numerical solution of the Poisson equation~\eqref{eq:poisson_equation_general_prototype} using ``deterministic'' techniques become preferable.
This is particularly the case when the structure of the fast-slow system~\eqref{eq:prototype_general_fast-slow_system} is such that
spectral methods can be applied in a straightforward manner.
Such an approach was taken in \cite{goudonefficient} for the study of
the diffusion approximation of a kinetic model for swarming~\cite{carillo2010}.
In dimensionless variables, the equation for the distribution function
$f^{\varepsilon}(x,v,t)$ reads
\begin{equation}
    \label{eq:kinetic_equation}
    \frac{\partial f^{\varepsilon}}{\partial t} + \frac{1}{\sqrt{\varepsilon}} (v \cdot \nabla_r f^{\varepsilon}
    - \nabla_r \Psi \cdot \nabla_v f^{\varepsilon}) = \frac{1}{\varepsilon}Q(f^{\varepsilon}),
\end{equation}
where
$\Psi$ is a potential that is defined self-consistently through the solution of a Poisson equation,
$Q(\cdot)$ denotes a linearized ``collision'' operator,
with the appropriate number and type of collision invariants.
It was shown in~\cite{goudonefficient} that in the limit as $\varepsilon$ tends to $0$,
the spatial density $\rho(x,t) = \int f(x,v,t) \, dv$ of swarming particles converges to
the solution of an aggregation-diffusion equation of the form
\begin{equation}
    \label{eq:agreggation-diffusion_swarming}
    \frac{\partial \rho}{\partial t} - \nabla \cdot (\mathcal D \nabla \rho + \mathcal K (\nabla U \star \rho)\rho) = 0,
\end{equation}
where  $\star$ denotes the convolution product,
$U$ is the interaction potential, and the drift and diffusion tensors $\mathcal K$ and $\mathcal D$,
respectively, can be calculated using an approach identical to~\eqref{eq:poisson_equation_general_prototype} and~\eqref{eq:effective_drift_and_diffusion}:
we first have to solve the Poisson equations\footnote{
    We first perform a unitary transformation that maps
    the generator of a diffusion process of the form $\mathcal L_y$ that appears in~\eqref{eq:poisson_equation_general_prototype}
    to an appropriate Schr\"odinger-type operator; see~\cite[Sec. 4.9]{pavliotis2011applied} for details.
}
\begin{equation}
    \label{eq:poisson_swarming}
    - \operator Hu_\chi = v\sqrt{M} \quad \text{ and } \quad
    - \operator Hu_\kappa = \frac{1}{\theta} \nabla_v W \sqrt{M},
\end{equation}
where $W(\cdot)$ is a potential in velocity,
$M(v) = Z^{-1} e^{-W(v)/\theta}$ is the Maxwellian distribution at temperature $\theta$,
with $Z$ being the normalization constant,
$\mathcal H = - \theta \Delta_v + \Phi(v)$ and
\begin{equation}
    \label{eq:potential_velocity_swarming}
    \Phi(v) = -\frac{1}{2} \Delta_v W(v) + \frac{1}{4 \theta} \abs{\nabla_v W(v)}^2.
\end{equation}
Then the effective coefficients can be calculated by the integrals
\begin{equation}
    \label{eq:drift_and_diffusion_swarming}
    \mathcal D = \int_{\real^d} \mathcal H (u_\chi) \otimes u_\chi \, dv \quad \text{ and } \quad
    \mathcal K = \int_{\real^d} \mathcal H (u_\chi) \otimes u_\kappa \, dv.
\end{equation}
We note that the operator $\mathcal H$ that appears in~\eqref{eq:poisson_swarming} is
a Schr\"{o}dinger operator whose spectral properties are very well understood~\cite{reed1978analysis, HislopSigal1996}.
In particular, under appropriate growth assumptions on the potential $\Phi$ given in~\eqref{eq:potential_velocity_swarming},
the operator $\mathcal H$ is essentially selfadjoint, has discrete spectrum and
its eigenfunctions form an orthonormal basis in $\lp{2}{\real^d}$.
The computational methodology that was introduced and analyzed in~\cite{goudonefficient}
for calculating the homogenized coefficients in~\eqref{eq:agreggation-diffusion_swarming}
is based on the numerical calculation of the eigenvalues and eigenfunctions of the Schr\"{o}dinger operator
using a high-order finite element method.
It was shown rigorously and by means of numerical experiments that
for sufficiently smooth potentials the proposed numerical scheme performs extremely well;
in particular, the numerical calculation of the first few eigenvalues and eigenfunctions of $\mathcal H$
are sufficient for the very accurate calculation of the drift and diffusion coefficients given in~\eqref{eq:drift_and_diffusion_swarming}.

In this paper we develop further the methodology introduced in~\cite{goudonefficient} and
we apply it to the numerical solution of fast/slow systems of SDEs,
including singularly perturbed stochastic partial differential equations (SPDEs) in bounded domains.
Thus, we complement the work presented in~\cite{abdulle2012numerical},
in which a hybrid HMM/spectral method for the numerical solution of singularly perturbed SPDEs with quadratic nonlinearities~\cite{blomker2007} at the diffusive time scale was developed.\footnote{
    When the centering condition (see Equation~\eqref{eq:assumption_f_bounded_polynomial}) is not satisfied,
    one needs to study the problem at a shorter time scale (called the advective time scale).
    This problem is easier to study since it does not require the solution of a Poisson equation.
    The rigorous analysis of the HMM method for singularly perturbed SPDEs at the advective time scale was presented in~\cite{brehier2013}.
}
The main difference between the methodology presented in~\cite{goudonefficient} and the approach we take in this paper is that,
rather than obtaining the orthonormal basis by solving the eigenvalue problem for an appropriate Schr\"odinger operator,
we fix the orthonormal basis (Hermite functions) and expand the solution of the Poisson equation~\eqref{eq:poisson_equation_general_prototype}
(after the unitary transformation that maps it to an equation for a Schr\"{o}dinger operator) in this basis.
We show rigorously and by means of numerical experiments that our proposed methodology achieves spectral convergence
for a wide class of fast processes in~\eqref{eq:prototype_general_fast-slow_system}.
Consequently, our method outperforms Monte Carlo-based methodologies such as the HMM and the equation-free method,
at least for problems with low-dimensional fast processes.
We discuss how our method can be modified so that
it becomes efficient when the fast process has a high-dimensional state space in the
conclusions section, \cref{sec:conclusion_and_perspectives_for_future_work}.

In this paper we will consider fast/slow systems of SDEs for which the fast process is reversible,
i.e. it has a gradient structure~\cite[Sec. 4.8]{pavliotis2011applied}\footnote{
    We could, in principle, also consider reversible SDEs with a diffusion tensor that is not a multiple of the identity.
}
\begin{subequations}
    \label{eq:multiscale_with_gradient_structure}
    \begin{eqnarray}
        \label{e:x}
        dX_t^\varepsilon & = & \frac{1}{\varepsilon} \vect f(X_t^\varepsilon,Y_t^\varepsilon)dt
        + \vect \alpha(X_t^\varepsilon,Y_t^\varepsilon)\,dW_{xt}, \quad X_0^\varepsilon = x_0, \\
        \label{e:y}
        dY_t^\varepsilon  & = & -\frac{1}{\varepsilon^2} \grad V(Y_t^\varepsilon)dt
        + \frac{\sqrt{2}}{\varepsilon} d W_{yt}, \quad Y_0^\varepsilon = y_0,
    \end{eqnarray}
\end{subequations}
where
$X_t^\varepsilon(t) \in \real^{m}$, $Y_t^\varepsilon(t)\in\real^{n}$,
$\vect \alpha(\cdot,\cdot) \in \real^{m\times p}$,
$W_{x}$ and $W_{y}$ are standard $p$ and $n$-dimensional Brownian motions,
and
$V(\cdot)$ is a smooth confining potential.
SDEs of this form appear in several applications, e.g. in molecular dynamics~\cite{DuncanKalliadasisPavliotisPradas2016,MR2681239}.
Furthermore, several interesting semilinear singularly perturbed SPDEs can be written in this form, see \cref{sec:numerics}.
It is well known~\cite[Sec. 4.9]{pavliotis2011applied} that
the generator of a reversible SDE is unitarily equivalent to an appropriate Schr\"{o}diner operator.
Consequently, the calculation of the drift and diffusion coefficients in the homogenized equation corresponding to~\eqref{eq:multiscale_with_gradient_structure} reduces to
a problem that is very similar to~\eqref{eq:poisson_swarming} and~\eqref{eq:drift_and_diffusion_swarming}.
Our approach is to first solve this Poisson equation for the Schr\"odinger operator via a spectral method using Hermite functions and then
use this solution in order to calculate the integrals in~\eqref{eq:effective_drift_and_diffusion}.
For smooth potentials that increase sufficiently fast at infinity our method has spectral accuracy,
i.e. the error decreases faster than any negative power of the number of floating point operations performed.
This, in turn, via a comparison for SDEs argument, implies that
we can approximate very accurately the evolution of observables of the slow variable $X_t^{\varepsilon}$ in~\eqref{eq:multiscale_with_gradient_structure}
by solving an approximate homogenized equation in which the drift and diffusion coefficients are calculated using our spectral method.
For relatively low dimensional fast-processes,
this leads to a much more accurate and computationally efficient numerical method than any Monte Carlo-based methodology.
We remark that our proposed numerical methodology becomes (analytically) exact when
the fast process is, to leading order, an Ornstein-Uhlenbeck process,
since in this case, for a suitable choice of the mean and the covariance matrix,
the Hermite functions are the eigenfunctions of the corresponding Schr\"odinger operator.

The rest of the paper is organized as follows.
In \cref{sec:diffusion_approximation}, we summarize the results from homogenization theory for the fast/slow system~\eqref{eq:multiscale_with_gradient_structure} that we will need in this work.
In \cref{sec:numerical_method} we present our numerical method in an algorithmic manner.
In \cref{sec:main_results}, we summarize the main theoretical results of this paper;
in particular we show that our method, under appropriate assumptions on the coefficients of the fast/slow system, is spectrally accurate.
The proofs of our main results are given in \cref{sec:analysis_of_the_method}.
In \cref{sec:numerics} we present details on the implementation of our numerical method,
discuss the computational efficiency and present several numerical examples,
including an example of the numerical solution of a singularly perturbed SPDE;
for this example, we also present a brief qualitative comparison of our method with the HMM method.
\Cref{sec:conclusion_and_perspectives_for_future_work} is reserved for conclusions and discussion of further work.
Finally in the appendices we present some results related to approximation theory in weighted Sobolev spaces
that are needed in the proof of the main convergence theorem.


\section{Diffusion Approximation and Homogenization}
\label{sec:diffusion_approximation}
In this section, we summarize some of our working hypotheses and the results from the theory of homogenization used to derive the effective SDE for the system \eqref{eq:multiscale_with_gradient_structure}.
Throughout this paper, the notation $\abs{\cdot}$ denotes the Euclidian norm when applied to vectors, and the Frobenius norm when applied to matrices.
In addition, for a vector $\vect{v} \in \real^d$, the components are denoted by $v_1, v_2 \cdots, v_d$.
We start by assuming that $V(\cdot)$ is a smooth confining potential,~\cite[Definition 4.2]{pavliotis2011applied}:
\begin{equation}\tag{$H_V$}
    \label{eq:assumptions_potential}
    V \in \smooth{\real^n}, \qquad \lim_{\abs{y} \to \infty} V(y) = \infty
    \quad \text{ and } \quad e^{-V(\cdot)} \in \lp{1}{\real^n}.
\end{equation}
These hypotheses guarantee that the fast process has a well defined solution for all positive times,
with a unique invariant measure whose density is given by $\frac{1}{\mathcal Z} e^{-V(y)}$,
where $\mathcal Z$ is the normalization constant.
Without loss of generality, we may assume that $\mathcal Z = 1$.
To these assumptions, we add
\begin{equation}\tag{$H_W$}
    \label{eq:assumption_w}
    \lim_{\abs{y}\to \infty} \grad{V}\cdot y = \infty \quad \text{ and } \quad
    \lim_{\abs{y}\to \infty} W(y) := \lim_{\abs{y}\to \infty} \left( \frac{1}{4} |\grad V(y)|^2 \,-\, \frac{1}{2}\Delta V(y) \right)= \infty,
\end{equation}
which guarantee that the law of $y(t)$ converges to its invariant distribution $e^{-V}$ exponentially fast (e.g. in relative entropy), see \cite{markowich2000}.
We assume furthermore that the drift coefficient in the slow equation of system \eqref{eq:multiscale_with_gradient_structure} satisfies
\begin{equation}\tag{$H_f$}
    \begin{aligned}
        &\vect f(x,y) \in \left(\smooth{\real^m \times \real^n}\right)^m,\\
        &\int_{\real^n} \vect f(x,y) \,  e^{-V(y)} \,dy = \vect 0,\text{ and }\\
        & \abs{\vect f(x,y)} \leq p(y) \quad \forall  x \in \real^m \text{ and }  \forall y \in \real^n,
    \end{aligned}
    \label{eq:assumption_f_bounded_polynomial}
\end{equation}
where $p(\cdot)$ is a polynomial.
Under Assumptions~\eqref{eq:assumptions_potential} and~\eqref{eq:assumption_f_bounded_polynomial},
the uniform ellipticity of the generator of the fast dynamics and \cite[Theorem 1]{pardoux2001poisson}
ensure that there exists for all $x\in \real^m$ a solution that is smooth in $y$ of the Poisson equations:
\begin{equation}
    \label{eq:poisson_equation}
    -\operator{L} \phi_i(x,y) := -\left( \Delta_y - \grad_y{V}\cdot \nabla_y \right) \phi_i(x,y)= f_i(x,y) \quad \text{ for } i = 1, \dots, m.
\end{equation}
The difference in sign was adopted to lighten the notation in the analysis presented in \cref{sec:analysis_of_the_method}.
We consider solutions that are locally bounded and grow at most polynomially in $y$.
The solution to the Poisson equations~\eqref{eq:poisson_equation} are unique, up to constants.
Without loss of generality, we can set these constants to be equal to $0$:
\begin{equation}
    \label{eq:mean_zero_solution_poisson_equation}
    \int_{\real^n} \vect \phi(x,y) \,  e^{-V(y)} \,dy = \vect 0, \quad \forall x \in \real^m.
\end{equation}
In addition to the previous assumptions,
we add the following assumption on the Lipschitz continuity with respect to $x$ of the coefficients.
\begin{equation}\tag{$H_L$}
    \abs{\vect f(x,y) - \vect f(x',y)} + \abs{\vect \alpha(x,y) - \vect \alpha(x',y)} \leq C(y) \abs{x - x'},
    \label{eq:assumption_lip}
\end{equation}
and the following assumptions on the growth of the coefficients:
\begin{equation}\tag{$H_G$}
    \begin{aligned}
        &\abs{\vect f(x,y)} \leq K(1+|x|)(1+|y|^{m_1}),\\
        &\abs{\nabla_x \vect f(x,y)} + \abs{\nabla_x^2 \mathbf f(x,y)} \leq K(1+\abs{y}^{m_2}),\\
        & \abs{\vect \alpha(x,y)} \leq K(1+|x|^{1/2})(1 + \abs{y}^{m_3}), \\
    \end{aligned}
    \label{eq:assumption_growth}
\end{equation}
for positive integers $m_1,m_2,m_3$ and a positive constant $K$.
It follows from this that $\vect \phi(\cdot,y)$ belongs to $\left(\cont{2}{\real^m}\right)^m$ for all values of $y$.
This can be shown by using the Feynman-Kac representation of the solution of~\eqref{eq:poisson_equation}
that was studied in \cite{pardoux2001poisson}:
\begin{equation}
    \label{eq:feynman-kac_formula}
    \phi_i(x,y) = \int_{0}^\infty \expect_y f_i(x,z^y_t) \, dt, \qquad i = 1,\dots,m,
\end{equation}
where $z^y_t$ is the solution of
\begin{equation*}
    dz^y_t = -\grad_y{V}(z_t^y) \, dt + \sqrt{2}\, dW_t \quad \text{ with } \quad z^y_0 = y.
\end{equation*}
Using the Feynman-Kac formula~\eqref{eq:feynman-kac_formula}, one can show~\cite[p. 1073]{pardoux2001poisson} that there exist $L,q > 0$ such that:
\begin{equation}
    \begin{aligned}
        & \abs{\vect \phi(x,y)} + \abs{\nabla_y \vect \phi(x,y)} \leq L(1 + \abs{x})(1 + \abs{y}^q), \\
        & \abs{\grad_x\vect \phi(x,y)} + \abs{\grad_y\grad_x \vect \phi(x,y)} + \abs{\grad_x \grad_x \vect \phi(x,y)} + \abs{\grad_y\grad_x \grad_x \vect \phi(x,y)} \leq L(1 + \abs{y}^q).
    \end{aligned}
    \label{eq:bounds_derivatives_phi}
\end{equation}
Using the previous assumptions we can prove the following homogenization/diffusion approximation result \cite[Theorem 3]{pardoux2001poisson}.
\begin{theorem}
    Let~\cref{eq:assumptions_potential,eq:assumption_w,eq:assumption_f_bounded_polynomial},~\eqref{eq:assumption_lip} and~\eqref{eq:assumption_growth} be satisfied.
    Then for any $T>0$,
    the family of processes $\left\{X^\varepsilon_t, 0 \leq t \leq T\right\}$ solving~\eqref{eq:multiscale_with_gradient_structure}
    is weakly relatively compact in $\left(\cont{}{[0,T]}\right)^m$.
    Any accumulation point $X_t$ is a solution of the martingale problem associated to the operator:
    \[
        \operator{G} = \frac{1}{2}\vect D(x):\grad_y \grad_y + \vect F(x)\cdot \grad_x
    \]
    where
    \begin{equation}
        \vect F(x) = \int_{\real^n} \grad_x{\vect \phi(x,y)} \, \vect f(x,y) \,e^{-V(y)}\,dy,
        \label{eq:definition_of_F}
    \end{equation}
    and
    \begin{equation}
        \vect D(x) = \int_{\real^n} \left(\vect \alpha(x,y) \vect \alpha(x,y)^T + \vect f(x,y) \otimes \vect \phi(x,y) + \vect \phi(x,y) \otimes \vect f(x,y)\right) \,  e^{-V(y)}\,dy,
        \label{eq:definition_of_D}
    \end{equation}
    where $\vect \phi(x,y)$ is the centered solution of the Poisson equation~\eqref{eq:poisson_equation}.
    If, moreover, the martingale problem associated to $\operator{G}$ is well-posed, then $X_t^\varepsilon \Rightarrow X_t$ (convergence in law),
    where $X_t$ is the unique diffusion process (in law) with generator $\operator{G}$.
\end{theorem}

In view of this theorem,
writing $\vect D(x)=\vect A(x)\vect A(x)^T$ we obtain the functions $\vect F(x)$, $\vect A(x)$ that appear in the homogenized SDE~\eqref{eq:simplified_equation_for_general_prototype}.


\section{Numerical Method}
\label{sec:numerical_method}
In this section, we describe our method for the approximation of the effective dynamics, the analysis of which is postponed to \cref{sec:analysis_of_the_method}.
We start by introducing the necessary notation.
We will denote by $\lp{2}{\real^n}$ the space of square integrable functions on $\real^n$, by $\ip{\cdot}{\cdot}[0]$ the associated inner product, and by $\norm{\cdot}[0]$ the associated norm.
The notation $\lp{2}{\real^n}[\rho]$, for a probability density $\rho$, will be used to denote the space of functions $f$ such that $\sqrt\rho f \in \lp{2}{\real^n}$.
Weighted Sobolev spaces associated to a probability density are defined in \cref{definition:weighted_sobolev_spaces_probability_density}.
whereas scales of Sobolev spaces, associated to an operator, are defined in \cref{definition:weighted_sobolev_spaces_operator}.

In addition to these function spaces, we will denote by $\poly{d}(\real^n)$ the space of polynomials in $n$ variables of degree less than or equal to $d$,
and by $\hermite H_{\alpha}(y; \mu,\Sigma)$ the Hermite polynomials on $\real^n$ defined in \cref{sec:hermite}:
\begin{equation}
    \label{def:vector_valued_Hermite_polynomials}
    \hermite H_{\alpha}(y; \mu,\Sigma)\,=\, \hermite H_{\alpha}^* (S^{-1}(y-\mu)), \quad \text{ with }
    \alpha \in \nat^n \text{ and } \hermite H_\alpha^*(z)={\prod}_{k\,=\,1}^n H_{\alpha_k}(z_k).
\end{equation}
Here $H_{\alpha_k}(\cdot)$ denotes one-dimensional Hermite polynomial of degree $\alpha_k$,
$\Sigma \in \real^{n\times n}$ is a symmetric positive definite matrix,  $D$ and $Q$ are diagonal and orthogonal matrices such that $\Sigma = Q D Q^T$, $S = Q D^{1/2}$ and $\mu\in \real^n.$
We recall from~\cref{sec:hermite} that these polynomials form a complete orthonormal basis of $L^2(\real^n, \gaussian[\mu][\Sigma])$,
where $G_{\mu,\Sigma}$ denotes the Gaussian density on $\real^n$ with mean $\mu$ and covariance matrix $\Sigma$.
Finally, we will use the notation $\hermitef{h}_{\alpha}(y;\mu,\Sigma)$ to denote the Hermite functions corresponding to the Hermite polynomials~\eqref{def:vector_valued_Hermite_polynomials}, see~\cref{definition:hermite_function}.

We recall from~\cref{sec:diffusion_approximation} that obtaining the drift and diffusion coefficients $\vect F(X)$ and $\vect A(X)$, respectively, of the homogenized equation
\begin{equation}
    \label{eq:exact_effective_equation}
    dX = \vect F(X) \, dt + \vect A(X)\,dW_t,
\end{equation}
requires the solution of the Poisson equations~\eqref{eq:poisson_equation}.
To emphasize the fact that $x$ appears as a parameter in~\eqref{eq:poisson_equation},
we will use the notations $\vect \phi^x(\cdot) := \vect \phi(x,\cdot)$ and $\vect f^x(\cdot) := \vect f(x,\cdot)$.
The weak formulation of the Poisson equation~\eqref{eq:poisson_equation} is to find $\vect \phi^x\in \sobolev{1}{\real^n}[e^{-V}]$ such that for $i = 1, \dots, m$,
\begin{align}
    \label{eq:weak_form_pde_cell_problem}
    a_V(\phi_i^x,v):= \int_{\real^n} \grad{\phi_i^x} \cdot \grad{v} \, e^{-V} dy = \int_{\real^n} f_i^x \, v\, e^{-V}dy \qquad \forall v \in \sobolev{1}{\real^n}[e^{-V}],
\end{align}
with the centering condition
\begin{equation}
    \label{eq:centering}
    \mathcal M(\vect \phi^x) := \int_{\real^n} \vect \phi^x \, e^{-V}\,dy = \vect 0.
\end{equation}
We recall that in order to be well-posed the condition $\mathcal M(\vect f^x)=0$ must be satisfied.

We start by performing the standard unitary transformation that maps the generator of a reversible Markov process to a Schr\"odinger operator: $e^{-V/2}: \lp{2}{\real^n}[e^{-V}] \to \lp{2}{\real^n}$.
Introducing
\begin{equation}
    \operator H := e^{-V/2} \operator L \left(e^{V/2} \cdot \right) = \Delta - \left(\frac{1}{4}|\grad V|^2 - \frac{1}{2}\Delta V\right)\,= \Delta - W(y),
    \label{def:H}
\end{equation}
and $\vect \psi^x=e^{-V/2} \vect \phi^x$,
the Poisson equation~\eqref{eq:poisson_equation} can be rewritten in terms of the operator~\eqref{def:H} as:
\begin{equation}
    \label{eq:cell_problem_mapped}
    - \operator H \vect \psi^x = e^{-V/2}\vect f^x.
\end{equation}
The weak formulation of this mapped problem reads: find $\vect \psi^x\in \sobolev{1}{\real^n}[\operator{H}]$
satisfying $\hat{\mathcal{M}}(\vect \psi^x):= \int_{\real^n} \vect \psi^x \, e^{-V/2} dy = 0$ and such that, for $i = 1, \dots, m$,
\begin{align}
    a(\psi^x_i,v):= \int_{\real^n} \grad{\psi_i^x} \cdot \grad{v} +W(y)\,\psi^x_i\, v\, dy = \int_{\real^n} f_i^x \, v\, e^{-V/2}dy \qquad \forall v \in  \sobolev{1}{\real^n}[\operator{H}],
    \label{eq:weak_form_pde_cell_problem_h}
\end{align}
where $\sobolev{1}{\real^n}[\operator{H}]= \left\{ u \in \sobolev{1}{\real^n} : \int_{\real^n} \abs{W} u^2 \, dy < \infty\right\}$.
The centering condition becomes:
\begin{equation}
    \label{eq:centering_hat}
    \hat{\mathcal{M}}(\vect \psi^x):= \int_{\real^n} \vect \psi^x e^{-V/2} dy = 0.
\end{equation}
The formulas for the effective drift and diffusion coefficients can be written as
\begin{subequations}
    \label{eq:effective_coefficients_after_unitary_mapping}
    \begin{align}
        &\vect F(x) = \int_{\real^n} \grad_x \vect \psi^x \, \left(\vect f^x \, e^{-V/2}\right)dy, \\
        &\vect D(x) = \int_{\real^n} \vect \alpha \vect \alpha^T(x,y)\,\mu^x(dy) + \vect A_0(x) + \vect A_0(x)^T,
    \end{align}
\end{subequations}
where
\begin{equation}
    \label{def:A_0}
    \vect A_0(x) = \int_{\real^n} \vect \psi^x \otimes \left(\vect f^x \, e^{-V/2}\right)\, dy.
\end{equation}
The advantage of using the unitary transformation is that the solution of this new problem and its derivative lie in $\lp{2}{\real^n}$, rather than in a weighted space.

To approximate numerically the coefficients of the effective SDE,
we choose a finite-dimensional subspace $\hat S_d$ of $\sobolev{1}{\real^n}[\operator{H}]$, specified below, and consider the finite-dimensional approximation problem:
find $\vect \psi_d^{x}\in \hat S_d$ such that, for $i = 1, \dots, m$,
\begin{align}
    a(\psi_{di}^{x},v_d)= \int_{\real^n} f^x_i \, v_d\, e^{-V/2}dy \qquad \forall v_d \in \hat S_d.
    \label{eq:weak_form_pde_cell_problem_num}
\end{align}
While the centering condition for $\vect \psi^x$ serves to guarantee the uniqueness of the solution to \eqref{eq:weak_form_pde_cell_problem_h},
it does not affect the coefficients \eqref{eq:definition_of_F}, \eqref{eq:definition_of_D} of the simplified equation.
Existence and uniqueness---possibly up to a function in the kernel of $\operator{H}$---of
the solution of the finite-dimensional problem are inherited from the infinite-dimensional problem~\eqref{eq:weak_form_pde_cell_problem_h}.

For a given basis ${\left\{e_\alpha\right\}}_{|\alpha| \leq d}$ of $\hat S_d$,
the finite-dimensional approximation of $\vect \psi^x$ can be expanded as $\vect \psi_d^{x} = \sum_{\abs{\alpha} \leq d} \vect \psi^x_\alpha \, e_\alpha$,
and from the variational formulation~\eqref{eq:weak_form_pde_cell_problem_num} we obtain the following linear systems:
\begin{equation}
    \label{eq:finite_dimensional_problem_in_matrix_form}
    \sum_{|\beta| \leq d} a(e_\alpha, \, e_\beta)\, \vect \psi^x_\beta = \vect f^x_\alpha \quad \text{with} \quad \vect f^x_\alpha = \int_{\real^n}\vect f^x\,e_\alpha\,e^{-V/2}\,dy.
\end{equation}
We will use the notation $A_{\alpha \beta} = a(e_\alpha, e_\beta)$ for the stiffness matrix.
In view of formula~\eqref{eq:effective_coefficients_after_unitary_mapping} we see that we also need
an approximation the gradient of the solution, which we denote by $\nabla_x \vect \psi_d^{x}$.
This can be obtained by solving~\eqref{eq:finite_dimensional_problem_in_matrix_form} with the right-hand side  $(\nabla_x \vect f^x)_\alpha = \int_{\real^n}(\grad_x \vect f^x)\,e_\alpha\,e^{-V/2}\,dy$.

Once the solutions $\vect \psi_d^{x}$ and $\nabla_x\vect \psi_d^{x}$ are computed, we can calculate the approximate drift and diffusion as follows.
Then, by substituting the approximations of $\vect \psi_d^{x}$, $\nabla_x\vect \psi_d^{x}$, and $e^{-V/2}\vect f^x$ in~\eqref{eq:effective_coefficients_after_unitary_mapping}, we obain
\begin{subequations}
    \begin{align}
        & \vect F_d(x) = \sum_{\abs{\alpha} \leq d} \sum_{\abs{\beta} \leq d} \ip{e_\alpha}{e_\beta}[0] \,(\nabla_x\vect \psi^x)_\alpha \cdot \vect f_\beta^x,\\
        & \vect A_{0d}(x) = \sum_{\abs{\alpha} \leq d} \sum_{\abs{\beta} \leq d} \ip{e_\alpha}{e_\beta}[0] \, \vect \psi_\alpha^x \otimes \vect f_\beta^x,\\
        & \vect D_d(x) = \int_{\real^n} \vect \alpha\vect  \alpha^T(x,y)\, e^{-V} dy + \vect A_{0d}(x)+  \vect A_{0d}(x)^T,\quad \vect A_d(x){\vect A_d(x)}^T= \vect D_d(x).
    \end{align}
    \label{eq:definition_of_fd_and_ad}
\end{subequations}
Using these coefficients, we obtain the approximate homogenized SDE
\begin{equation}
    \label{eq:approximate_effective_equation}
    dX_d \,=\, \vect F_d(X_d)dt + \vect A_d(X_d) dW_t.
\end{equation}
This equation can now be easily solved using a standard numerical method, e.g. Euler-Maruyama.

Our numerical methodology is based on the expansion of the solution to~\eqref{eq:cell_problem_mapped} in Hermite functions:
\begin{equation}
    \label{eq:finite_dimensional_space_of_galerkin_method}
    \hat S_d=\hbox{span}\{ \hermitef{h}_{\alpha}(y;\mu,\Sigma)\}_{|\alpha|\leq d}.
\end{equation}
%
%
%
%
%
%
A good choice of the mean and covariance, $\mu$ and $\Sigma$, respectively, is important for the efficiency of the algorithm.
In our implementation we choose
\begin{equation}
    \label{eq:choice_of_mu_and_sigma}
    \mu = \int_{\real^n} y\, e^{-V(y)}dy\quad \text{ and } \quad \Sigma =  \lambda \int_{\real^n} (y-\mu){(y-\mu)}^T\,e^{-V(y)}\,dy,
\end{equation}
where $\lambda > 0$ is a free parameter independent of the first two moments of $e^{-V}$.
This choice for the mean and covariance  guarantees that our method is invariant under the rescaling $\tilde Y_t^\varepsilon = \sigma(Y_t^\varepsilon-m)$.
An example illustrating why this is desirable is when the mass of the probability density $e^{-V}$ is concentrated far away from the origin.
Using Hermite functions centered at 0 would provide a very poor approximation in this case, but choosing Hermite functions around the center of mass of $e^{-V}$ leads to a much better approximation.
Note that this is not the only choice that guarantees invariance under rescaling, but it is the most natural one.
\begin{remark}
    \label{remark:scaling_parameter}
    When the potential $V$ is quadratic, say $V(y) = \frac{1}{2} {(y-m)}^T S (y-m)$, the eigenfunctions of the operator $\operator H$ (defined in \eqref{def:H}) are precisely the Hermite functions $\hermitef{h}_\alpha(y;m,S)$.
    Hence choosing these as a basis, i.e. $e_\alpha = \hermitef{h}_\alpha(y;m,S)$, leads to a diagonal matrix $A$ in the linear systems~\eqref{eq:finite_dimensional_problem_in_matrix_form},
    because $a(e_\alpha, e_\beta) = \lambda_\alpha \delta_{\alpha \beta}$, with $\lambda$ defined in \cref{eq:eigenvalues_of_hermite polynomials}.
    This choice corresponds to $\lambda = 1$ in~\eqref{eq:choice_of_mu_and_sigma}.
    The optimal choice for the  parameters $\mu$ and $\Sigma$ for a general density $e^{-V}$ and function $\vect f$ has been partially studied.
    In particular, it was shown in \cite{gottlieb1977numerical} that $O(p^2)$ Hermite polynomials are necessary to resolve $p$ wavelengths of a sine function, when keeping the scaling parameter fixed.
    This result carries over to the case of normalized Hermite functions,
    where the associated covariance matrix would play the role of the scaling parameter.
    More recently, it was shown in~\cite{tang1993hermite} that much better results could be obtained by choosing the scaling parameter as a function of the degree of approximation.
    In particular, it was shown that that by choosing this parameter inversely proportional to the number of Hermite functions,
    only $O(p)$ functions are needed in order to resolve $p$ wavelengths in one spatial dimension.
\end{remark}
%
%
    \begin{framed}
    \paragraph{Summary of the Method}
    In short, the method can be summarized as follows.

    For a given initial condition $X^\varepsilon(0)=X_0$, $n=0,\,1,\,2,\,\ldots$,
    a given stochastic integrator $X_d^{n+1}=\Psi(X_d^{n},\vect F_d,\vect A_d,\Delta t,\xi_{n})$,
    and a chosen time step $\Delta t$, set $X_0^n = X_0$ and
     \begin{enumerate}
        \item Compute the solution $\vect \psi_d^{X_d^{n}}$ and $\nabla_x\vect \psi_d^{X_d^{n}}$ of~\eqref{eq:finite_dimensional_problem_in_matrix_form};
        \item Evaluate $\vect F_d(X_d^n),\vect A_d(X_d^{n})$ from~\eqref{eq:definition_of_fd_and_ad};
        \item Compute a time step $X_d^{n+1}=\Psi(X_d^{n},\vect F_d,\vect A_d,\Delta t,\xi_{n})$, and go back to 1.
     \end{enumerate}
    \end{framed}

\section{Main Results}
\label{sec:main_results}

In this section we present the main results on the analysis of our numerical method,
the proof of which will be presented in \cref{sec:analysis_of_the_method}.
We first need to introduce some new notations.
We will denote by $\ip{\cdot}{\cdot}[e^{-V}]$ the inner product of $\lp{2}{\real^n}[e^{-V}]$, defined by
$
    \ip{u}{v}[e^{-V}] = \int_{\real^n} u\,v\,e^{-V}\,dy,
$
and by $\norm{\cdot}[e^{-V}]$ the associated norm.
We will also use the notation $\norm{\cdot}[k][e^{-V}]$ for the norm of $\sobolev{k}{\real^n}[e^{-V}]$,
and $\norm{\cdot}[k][\operator{O}]$, where $\operator{O}$ is an operator, for the norm of $\sobolev{k}{\real^n}[\operator{O}]$,
see~\cref{sec:weighted_sobolev_spaces}.
We will denote by $\pi(\cdot)$ the projection onto mean-zero functions of $\lp{2}{\real^n}[e^{-V}]$, defined by
\begin{equation}
    \label{eq:projection_in_weighted_space}
    \pi(v) = v - \ip{v}{1}[e^{-V}], \quad v \in \lp{2}{\real^n}[e^{-V}].
\end{equation}
We will work mostly with the Schr\"odinger formulation~\eqref{eq:cell_problem_mapped} of the Poisson equation.
In that context,
we will employ the $\lp{2}{\real^n}$ projection operator on $\{\hat v \in \lp{2}{\real^n}: \hat M(\hat v) = 0\}$,
see~\cref{eq:centering_hat},
which we denote by $\hat \pi(\cdot)$ :
\begin{equation}
    \label{eq:projection_in_flat}
    \hat \pi(\hat v) = \hat v - \ip{\hat v}{e^{-V/2}}_0 \, e^{-V/2}, \quad \hat v \in \lp{2}{\real^n}.
\end{equation}

Finally, we will say that a function $g \in \lp{2}{\real^n} \cap \smooth{\real^n}$ decreases faster than any exponential function in the $\lp{2}{\real^n}$ sense if
\begin{equation}
    \label{e:exponential_decrease_L2}
    \int_{\real^n} {g(x)}^2 e^{\mu |y|} \, dy < \infty \quad \forall \mu \in \real,
\end{equation}
and denote by $E(\real^n)$ the space of all such functions.

In addition to the hypotheses presented in \cref{sec:diffusion_approximation}, we will employ the following assumptions.
\begin{assumption}
    \label{assumption:potential}
    The potential $W(y)$,
        introduced in~\eqref{eq:assumption_w}, is bounded from above by a polynomial of degree $4k$, for some $k \in \nat$.
        Furthermore, for every multi-index $\alpha$, there exist constants $c_\alpha > 0$ and $\mu_\alpha \in \real$ such that
        \[
            \abs{\partial_y^\alpha V} \leq c_\alpha \, e^{\mu_\alpha |y|},
        \]
        where $V(\cdot)$ is the potential that appears in~\eqref{e:y}.
\end{assumption}
\begin{assumption}
    \label{assumption:right-hand_side}
    The drift vector $\vect f(x,y)$ in~\eqref{e:x} is such that
        $e^{-V(\cdot)/2} \, \partial^\alpha_y \vect f(x,\cdot) \in \left(E(\real^n)\right)^m$ and
        $e^{-V(\cdot)/2} \, \partial^\alpha_y \nabla_x \vect f(x,\cdot) \in \left(E(\real^n)\right)^{m\times m}$
    for all $\alpha \in \nat^n$ and $x \in \real^m$.
\end{assumption}
For the proof of our main theorem we will need to have control on higher order derivatives of the solution to the Poisson equation~\eqref{eq:poisson_equation}.
To obtain such bounds we need to strengthen our assumptions on $\vect f(x,y)$ in~\eqref{e:x}.
In particular, in addition to~\eqref{eq:assumption_growth}, we assume the following:
\begin{assumption}
    \label{assumption:higher_moment_bounds}
    For all $\alpha \in \nat^n$, there exist constants $C_\alpha > 0$ and $\ell_\alpha \in \nat$ such that
    \begin{equation}
        \abs{\partial_y^\alpha \vect f} + \abs{\partial_y^\alpha \nabla_x \vect f} \leq C_\alpha \, (1 + \abs{y}^{\ell_\alpha}).
    \end{equation}
    In addition, the diffusion coefficient in the right-hand side of~\eqref{e:x} satisfies
    \begin{equation}
        \abs{\vect \alpha(x,y)} \leq K(1 + \abs{y}^{m_3}),
    \end{equation}
    for constants $K$ and $m_3$ independent of $x$.
\end{assumption}
From the Pardoux-Veretennikov bounds~\eqref{eq:bounds_derivatives_phi},
a bootstrapping argument, \cref{assumption:potential,assumption:higher_moment_bounds} and the integrability of monomials with respect to Gaussian weights we obtain the bounds
\begin{equation}
    \label{eq:bounds_on_weighted_norms}
    \norm{\vect \phi^x}[s,\operator L_{\mu,\Sigma}] \vee \norm{\nabla_x\vect  \phi^x}[s,\operator{L}_{\mu,\Sigma}] \vee \norm{\vect f^x}[e^{-V}] \leq C(s),
\end{equation}
for $s \in \nat$ and a constant $C(s)$ independent of $x$,
and where $a \vee b$ denotes the maximum between $a$ and $b$.
\Cref{assumption:higher_moment_bounds} and the moment bounds from~\cite{pardoux2001poisson} guarantee that
the coefficients of the homogenized equation~\eqref{eq:simplified_equation_for_general_prototype} are smooth and Lipschitz continuous.
Combined with the Poincaré inequality from~\eqref{proposition:poincare_inequality_weighted_measure},
they imply that the approximate coefficients calculated by~\eqref{eq:definition_of_fd_and_ad} are also globally Lipschitz continuous.
\begin{remark}
    In \cref{assumption:higher_moment_bounds} we assumed that the derivatives of the drift vector in~\eqref{e:x} with respect to $y$ are bounded uniformly in $x$.
    This is a very strong assumption and it can be replaced by a linear growth bound as in~\eqref{eq:assumption_growth}. Under such an assumption the proof of \cref{theorem:approximation_convergence_sde} has to be modified using a localization argument that is based on the introduction of appropriate stopping times.
    Although tedious, this is a standard argument, see e.g.~\cite{higham2002strong}, and we will not present it in this paper. Details can be found in~\cite{mres_urbain}.
\end{remark}
\begin{theorem}[Spectral convergence of the Hermite-Galerkin method]
\label{theorem:spectral_convergence_of_the_hermite_galerkin_method}
Under \cref{assumption:potential,assumption:right-hand_side},
there exists for all $x \in \real^m$ and $s \in \nat$ a constant $C(x,s)$ such that the
approximate solutions $\psi_d^x$ and $\nabla_x \psi_d^x$ satisfy the following error estimate:
\[
    \norm{\hat {\vect \pi}(\vect \psi^x_d) - \vect \psi^x}[0]
    \vee \norm{\hat \pi(\nabla_x \vect \psi^x_d) - \nabla_x \vect \psi^x}[0]
        \leq C(x,s) \, d^{-s}.
\]
\end{theorem}
Using this result, we can prove spectral convergence for the calculation of the drift and diffusion coefficients.

\begin{theorem}[Convergence of the drift and diffusion coefficients $F_d$ and $A_d$]
\label{theorem:convergence_of_drift_and_diffusion_coefficients}
Suppose that Assumptions~\ref{assumption:potential},~\ref{assumption:right-hand_side} and~\ref{assumption:higher_moment_bounds} hold.
Then the error on the approximate drift and diffusion coefficients decreases faster than any negative power of $d$, uniformly in $x$,
i.e. for all $s \in \nat$ there exists $D(s)$ such that
    \[
        \sup_{x \in \real^m}\abs{\vect F_d(x) -\vect F(x)} \vee \abs{\vect A_d(x)\,\vect A_d(x)^T - \vect A(x)\,\vect A(x)^T} \leq D(s)\, d^{-s}.
    \]
\end{theorem}

Using the spectral convergence of the approximate calculation of the drift and diffusion coefficients,
we can now control the distance between the solution of the homogenized SDE and its approximation~\eqref{eq:approximate_effective_equation}.
Denoting by $X(t)$ the exact solution of the homogenized equation and by $X_d(t)$ the approximate solution,
we use the following norm to measure the error:
\begin{equation}
    ||| X(t)  - X_d(t) ||| :=   \left( \expect \left[ \sup_{0\, \leq \,t\, \leq \,T} \, |X(t) \,-\,X_d(t)|^2 \right] \right)^{1/2}.
\end{equation}
\begin{theorem}
    \label{theorem:approximation_convergence_sde}
    Let \cref{assumption:potential,assumption:right-hand_side,assumption:higher_moment_bounds} hold.
    Then the error between the approximate and exact solutions of the simplified equation satisfies
    \begin{equation}
        |||X(t) \,-\,X_d(t)||| \, \leq \,\sqrt{4\,(T\,+\,4)\,D(s)\,T\,d^{-s}}\,\exp \left(2\,(T\,+\,4)\,C_L\, T\right),
    \end{equation}
    for any $s \in \nat$ and $T > 0$.
\end{theorem}
Now we consider the fully discrete scheme.
We need to consider an appropriate discretization of the approximate homogenized equation~\eqref{eq:approximate_effective_equation}.
For simplicity we present the convergence results for the case when we discretize the homogenized SDE using the Euler-Maruyama method:
\begin{equation}
    \label{eq:stochastic_integrator_euler-maruyama}
    X_d^{n+1}=X_d^n+\Delta t\, \vect F_d(X_d^n)+\vect A_d(X_d^n)\, \Delta W_n,
\end{equation}
but we emphasize that any higher order integrator, e.g. the Milstein scheme, could be used \cite{kloeden1992numerical,MiT04}.
The following is a classical result on the convergence of $X_d^n$ for which we refer to~\cite{kloeden1992numerical,MiT04,higham2002strong} for a proof.

\begin{theorem}[Convergence of the SDE solver]
    \label{theorem:convergence_of_euler-maruyama}
    Assume that $X_0$ is a random variable such that $\expect|X_0|^2<\infty$ and that \cref{assumption:potential,assumption:right-hand_side,assumption:higher_moment_bounds} hold. Then
    \begin{equation}
        \left(\expect \left[ \sup_{n\Delta t\in[0,T]} |X_d^n \,-\,X_d(t_n)|^2 \right] \right)^{\frac{1}{2}}\, \leq C(T) \sqrt{\Delta t}.
    \end{equation}
    for any choice of $T$, where $X_d^n$ denotes the solution of~\eqref{eq:stochastic_integrator_euler-maruyama}.
\end{theorem}
Combined, \cref{theorem:approximation_convergence_sde} and \cref{theorem:convergence_of_euler-maruyama} imply
the weak convergence of the solution of~\eqref{eq:stochastic_integrator_euler-maruyama} to
the solution of the homogenized equation~\eqref{eq:exact_effective_equation}.

\section{Proofs of the Main Results}
\label{sec:analysis_of_the_method}

\subsection{Convergence of the Spectral Method for the Poisson Equation}
\label{sub:convergence_of_the_solution_of_the_cell_problem}
In this section we establish the convergence of the spectral method for the solution of the Poisson equation \eqref{eq:weak_form_pde_cell_problem}.
Since the variable $x$ only appears as a parameter in the Poisson equation,
we will consider in this section that it takes an arbitrary value and will omit it from the notation.
Additionally, to disencumber ourselves of vectorial notations,
we will consider an arbitrary direction of $\real^n$, defined through a unit vector $\vect e$,
and denote by $f$ the projection $\vect f \cdot \vect e$.

We recall from \cite{pardoux2001poisson, pardoux2003poisson} that
there exists a unique smooth mean-zero function of $\phi \in \sobolev{1}{\real^n}[e^{-V}]$ satisfying the variational formulation
\begin{equation}
    \label{eq:variational_formulation_in_weighted_space}
    a_V(\phi,v) := \ip{\grad{\phi}}{\grad{v}}[e^{-V}] = \ip{f}{v}[e^{-V}] \quad \forall v \in \sobolev{1}{\real^n}[e^{-V}].
\end{equation}
We now define a finite-dimensional subset $S_d$ of $\sobolev{1}{\real^n}[e^{-V}]$ by $S_d = e^{V/2} \hat S_d$,
where $\hat S_d$ is the approximation space defined in eq.~\eqref{eq:finite_dimensional_space_of_galerkin_method},
and consider the following problem:
find $\phi_d\in S_d$ satisfying:
\begin{equation}
    \label{eq:approximate_cell_problem_weighted}
    a_V(\phi_d,v_d)= \ip{f}{v_d}[e^{-V}] \quad \forall v_d \in S_d.
\end{equation}
Note that, by definition of $f$,
$\phi = \vect \phi \cdot \vect e$ and $\phi_d = \vect \phi_d \cdot \vect e$.
The convergence of $\phi_d$ to $\phi$ can be obtained using techniques from the theory of finite elements,
in particular C\'ea's lemma and  an approximation argument.
We will use the notation that was introduced at the beginning \cref{sec:main_results}.
\begin{lemma}[C\'ea's lemma]
   \label{lemma:modified_cea}
    Let $\phi$ be the solution of~\eqref{eq:variational_formulation_in_weighted_space} satisfying $\mathcal M(\phi)=0$ and $\phi_d$ be a solution of~\eqref{eq:approximate_cell_problem_weighted}.
    Then,
    \[
        \norm{\phi - \pi(\phi_d)}[1][e^{-V}] \leq C \inf_{v_d \in S_d} \norm{\phi-v_d}[1][e^{-V}].
    \]
\end{lemma}
\begin{proof}
    The main ingredient of the proof is a Poincar\'e inequality for the measure $e^{-V} dx = \mu(dx)$ recalled in \cref{sec:weighted_sobolev_spaces}, \cref{proposition:poincare_inequality_weighted_measure}.
    From this inequality, we obtain the coercivity estimate $c \,a(v,v) \geq \norm{\pi(v)}[1][e^{-V}]^2$ for all $v \in \sobolev{1}{\real^n}[e^{-V}]$.
    Combining this with Galerkin orthogonality, $a(\phi-\phi_d, v_d) = 0$ for all $v_d \in S_d$ and the continuity estimate $a(v_1,v_2) \leq \norm{v_1}[1][e^{-V}] \norm{v_2}[1][e^{-V}]$ for all $v_1,v_2 \in \sobolev{1}{\real^n}[e^{-V}]$ gives the result.
\end{proof}

Since we will be working mostly with the Schr\"{o}dinger formulation of Poisson equation, we need an analogue of~\cref{lemma:modified_cea} for the transformed PDE.
We recall from \cref{sec:weighted_sobolev_spaces} that the space $\sobolev{1}{\real^n}[\operator{H}]$ is equipped with the norm
\[
    \norm{\psi}[1][\operator{H}]^{2}  = \norm{\psi}[0]^2 + \int_{\real^n} \abs{\grad{\psi}}^2 \,dy + \int_{\real^n} W \psi^2\,dy.
\]
\begin{lemma}
    \label{lemma:modified_cea_flat}
    Let $\vect \psi$ be the unique solution of~\eqref{eq:weak_form_pde_cell_problem_h} satisfying $\hat{\mathcal{M}}(\psi)=0$ and $\vect \psi_d$ be a solution of~\eqref{eq:weak_form_pde_cell_problem_num}.
    Then the projections $\psi = \vect \psi \cdot \vect e$ and $\psi_d = \vect \psi_d \cdot \vect e$ satisfy
    \begin{equation}
        \norm{\psi - \hat{\pi}(\psi_d)}[1][\operator{H}] \leq C \inf_{v_d \in \hat S_d} \norm{\psi-v_d}[1][\operator{H}].
        \label{eq:modified_cea_flat}
    \end{equation}
\end{lemma}
\begin{proof}
    The result follows directly by using the fact that $e^{-V/2}$ is also a unitary transformation from $\sobolev{1}{\real^n}[e^{-V}]$ to $\sobolev{1}{\real^n}[\operator{H}]$.
\end{proof}

Next, we focus on establishing a result that will allow us to control the right-hand side of~\eqref{eq:modified_cea_flat}.
In~\cite[Lemma 2.3]{gagelman2012spectral} the authors show that any smooth square integrable function such that $(-\Delta + W)v = g$ lies in the space $E(\real^n)$ introduced in~\eqref{e:exponential_decrease_L2},
provided that $g\in E(\real^n)$ and that Assumption~\eqref{eq:assumption_w} holds.
Differentiating the equation with respect to $y_i$, we obtain:
\[
    (-\Delta + W)\,\partial_{y_i} v = \partial_{y_i} g - \partial_{y_i}W\,v,
\]
so it is clear by \cref{assumption:potential} that $\partial_{\alpha} \psi \in E(\real^n)$ for all values of $\alpha \in \nat^n$.
This implies that $\psi$ belongs to the Schwartz space $\schwartz{\real^n}$.
We now generalize sligthly~\cite[Lemma 3.1]{gagelman2012spectral}.
This result will enable to control the norm $\norm{\cdot}[1][\operator{H}]$ on the right-hand side of~\eqref{eq:modified_cea_flat} by a norm $\norm{\cdot}_{k,\operator{H}_{\mu,\Sigma}}$,
where $\operator{H}_{\mu,\Sigma}$ is an operator defined in \cref{sec:weighted_sobolev_spaces}.
From this appendix, we recall that the operator $\operator H_{\mu,\Sigma}$,
with $\mu \in \real^n$ and $\Sigma$ a symmetric positive definite matrix,
is defined by $\operator{H}_{\mu,\Sigma} = -\Delta + W_{\mu,\Sigma}(y)$,
where $W_{\mu,\Sigma}$ denotes the quadratic function $(y-\mu)^T \Sigma^{-2} (y-\mu)/4 - \trace{\Sigma^{-1}}/2$.
\begin{lemma}
\label{lemma:bound_by_gaussian_norm}
    For every $k \in \nat$ and $v \in \schwartz{\real^n}$,
    \[
        \int_{\real^n}\abs{y}^{4k} v^2(y) \, dy \leq C(k, \mu, \Sigma) \norm{v}[2k][\operator{H}_{\mu,\Sigma}]^2,
    \]
    where $C(k,\mu,\Sigma)$ is a constant independent of $v$.
    \begin{proof}
        We set $Q_{\mu,\Sigma} = (y-\mu)^T \Sigma^{-2} (y-\mu)/4$.
        Following the methodology used to prove lemma 3.1 in \cite{gagelman2012spectral}, we establish that:
        \[
            \norm{Q_{\mu,\Sigma}(y)^{k+1} v}[0]^2 \leq \norm{Q_{\mu,\Sigma}(y)^{k}\, \left(\operator{H}_{\mu,\Sigma} + {\trace{\Sigma^{-1}}}/{2}\right) v}[0]^2  + C_1(k,\Sigma) \norm{Q_{\mu,\Sigma}(y)^{k}v}[0]^2,
        \]
        for all $k\in \nat$, and where $C_1(k,\Sigma) = (4k + 2)(k\,\rho(\Sigma^{-2}) + \trace{\Sigma^{-2}}/4)$.
        Reasoning by recursion and applying the triangle inequality, this immediately implies
        \[
            \begin{aligned}
                \norm{Q_{\mu,\Sigma}(y)^{k} v}[0]^2 &\leq \sum_{i=0}^{k} c_i(k,\Sigma)\, \norm{\left(\operator{H}_{\mu,\Sigma} + {\trace{\Sigma^{-1}}}/{2}\right)^i v}[0]^2  \\
                &\leq C_2(k,\Sigma) \norm{v}[2k][\operator{H}_{\mu,\Sigma}]^2,
            \end{aligned}
        \]
        To conclude, note that
        \[
            \abs{y}^{4k} \,\leq\, C_3 + C_4 \, Q_{\mu,\Sigma}(y)^{2k},
        \]
        for suitably chosen $C_3$ and $C_4$ depending on $\Sigma$ and $\mu$.
    \end{proof}
    \label{lemma:approximation_by_hermite_functions_in_other_weighted_space}
\end{lemma}
A finer version of the previous inequality could be obtained by following the argumentation of  in \cite[Theorem 3.2]{gagelman2012spectral},
but this will not be necessary for our purposes.
\Cref{lemma:bound_by_gaussian_norm} can be used to show the following result.
\begin{lemma}
If $W(y)$ is bounded above by a polynomial of degree $4k$, there exists a constant $C$ depending on $k$, $\mu$, $\Sigma$, and $W$ such that any $v \in \schwartz{\real^n}$ satisfies
\[
    \norm{v}[1][\operator{H}] \leq C \, \norm{v}[2k][\operator{H}_{\mu,\Sigma}].
\]

\end{lemma}
\begin{proof}
This follows from the considerations of \cref{sec:weighted_sobolev_spaces}.
First we note that
\[
    \norm{v}[1][\operator{H}]^2 = \norm{v}[1][\operator{H}_{\mu,\Sigma}]^2 + \int_{\real^n} (W - W_{\mu,\Sigma}) v^2\,dy.
\]
To bound the second term, we use~\cref{assumption:potential} on $W$, together with \cref{lemma:bound_by_gaussian_norm}:
\[
    \int_{\real^n} (W - W_{\mu,\Sigma}) v^2 \,dy\,\leq\, \int_{\real^n} (C_1 + C_2 \, \abs{y}^{4k}) v^2 \,dy \leq\, C_3 \norm{v}[2k][\operator{H}_{\mu,\Sigma}]^2,
\]
with $C_1$, $C_2$, $C_3$ depending on $k$, $\mu$, $\Sigma$.
\end{proof}

Upon combining the results presented so far in this section, we can complete the proof of \cref{theorem:spectral_convergence_of_the_hermite_galerkin_method}.

\begin{proof}[Proof of \cref{theorem:spectral_convergence_of_the_hermite_galerkin_method}]
    By \cref{lemma:modified_cea_flat,lemma:bound_by_gaussian_norm}, and the fact that the exact solution $\psi$ and its derivatives are smooth and decrease faster than exponentials, we have:
    \[
        \norm{\psi-\hat{\pi}(\psi_d)}[1][\operator{H}] \leq C \inf_{v_d \in \hat S_d} \norm{\psi-v_d}[1][\operator{H}] \leq C \inf_{v_d \in \hat S_d} \norm{\psi -v_d}[2k][\operator{H}_{\mu,\Sigma}].
    \]
    Using \cref{corollary:approximation_by_hermite_functions} on approximation by Hermite functions, we have for any $s > 2k$
    \begin{align*}
        \norm{\psi-\hat{\pi}(\psi_d)}[1][\operator{H}] & \leq C (d+1)^{-\frac{s-2k}{2}} \norm{\psi}[s][\operator{H}_{\mu,\Sigma}], \\
        & \leq \, C (d+1)^{-\frac{s-2k}{2}},
    \end{align*}
    where we used the first estimate of~\eqref{eq:bounds_on_weighted_norms}
    and the fact that $\norm{\psi}[s][\operator{H}_{\mu,\Sigma}]=\norm{\phi}[s][\operator{L}_{\mu,\Sigma}]$.
    The same reasoning can be applied to $\nabla_x \psi$.
    Since $s$ was arbitrary, this proves the statement.
\end{proof}

\subsection{Convergence of the Drift and Diffusion Coefficients}
\label{sub:convergence_of_the_coefficients}
In this section we prove the convergence of the drift and diffusion coefficients obtained from the approximate solution of the Poisson equation.
\begin{proof}[Proof of \cref{theorem:convergence_of_drift_and_diffusion_coefficients}]
From the expressions of $\vect F$ and $\vect F_d$ we have:
\begin{equation*}
    \vect F(x)-\vect F_d(x) = \int_{\real^n} \left[ \grad_x \vect \psi^x \cdot (\vect f^x \, e^{-V/2})-\grad_x \vect \psi_d^x \cdot (\vect f^x_d \, e^{-V/2}) \right] dy
\end{equation*}
where $\vect f_d^x \, e^{-V/2}$ is the $\lp{2}{\real^n}$-projection of $\vect f^x\,e^{-V/2}$ on the space spanned by Hermite functions with multi-index $\alpha$ such that $\abs{\alpha} \leq d$.
Clearly, $\int_{\real^n} \grad_x \vect \psi_d^x \cdot ( \vect f^x_d \, e^{-V/2} ) \, dy = \int_{\real^n} \grad_x \vect \psi_d^x \cdot ( \vect f^x \, e^{-V/2} )\,dy$,
and so using \cref{theorem:spectral_convergence_of_the_hermite_galerkin_method} together with the Cauchy-Schwarz inequality we deduce that there exists for any value of $s \in \nat$ a constant $C(s)$ such that
\begin{equation*}
    \begin{aligned}
        |\vect F_d(x)-\vect F(x)| &\leq \norm{\nabla_x \vect \psi^x - \nabla_x \vect \psi^x_d}[0]\,\norm{\vect f^x \, e^{-V/2}}[0] \\
        &\leq C(s) \, d^{-s} \norm{\vect f^x}[e^{-V}].
    \end{aligned}
\end{equation*}
The error on the diffusion term can be bounded similarly:
\begin{align*}
    \abs{\vect A_{0d}(x) - \vect A_0(x)} & = \int_{\real^n} (\vect \psi^x_d - \vect \psi^x) \otimes (\vect f^x \, e^{-V/2})\, dy \\
    & \leq C(s) \, d^{-s}\,\norm{\vect f^x}[e^{-V}].
\end{align*}
The proof can then be concluded using the last bound from~\eqref{eq:bounds_on_weighted_norms}.
\end{proof}

\subsection{Convergence of the Solution to the SDE}
\label{sub:strong_convergence_of_the_solution_of_the_sde}
As we have already mentioned, homogenization/diffusion approximation theorems are generally of the weak convergence type.
Furthermore, the effective diffusion coefficient of the simplified equation is not uniquely defined---see Equation~\eqref{eq:definition_of_D} and the fact that $\vect D(x) = \vect A(x)\vect A(x)^T $.
Consequently, it is not clear whether it is useful to prove the strong convergence of the solution to the approximate SDE~\eqref{eq:approximate_effective_equation} to the solution to the homogenized SDE~\eqref{eq:exact_effective_equation}.
However, by calculating $\vect A_d(x)$ by Cholesky factorization,
the difference $\abs{\vect A_d(x) - \vect A(x)}$ converges to $0$ faster than any negative power of $d$,
as is the case for $\abs{\vect A_d(x)\vect A_d(x)^T  - \vect A(x)\vect A(x)^T}$.
For this particular choice, it is possible to prove strong convergence of
solutions of~\eqref{eq:approximate_effective_equation} to the solution of~\eqref{eq:exact_effective_equation},
from which  weak convergence follows.
This is the approach taken in this section.

The argument we propose is based on the proof of the strong convergence for the Euler-Maruyama scheme in \cite[Theorem 2.2]{higham2002strong}.
Recall that by~\eqref{eq:bounds_on_weighted_norms}, there exists a Lipschitz constant $C_L$ such that
\begin{equation}
    \label{eq:lipschitz_result_for_proof_in_time}
    \left|\vect F(a)-\vect F(b)\right|^2\,\vee\,\left|\vect A(a)-\vect A(b)\right|^2\, \leq C_L|a-b|^2,
\end{equation}
for all $a,b \in \real^m$, and by \cref{theorem:convergence_of_drift_and_diffusion_coefficients}
there exists for every $s \in \nat$ a constant $D(s)$ independent of $d$ and $x$ such that
\begin{equation}
    \label{eq:convergence_coefficients_for_proof_in_time}
    \left|\vect F_d(x)-\vect F(x)\right|^2\,{\vee}\,\left|\vect A_d(x) - \vect A(x)\right|^2\,  \leq \, D(s) \,d^{-s},
\end{equation}
for any $x \in \real^m$.
Upon combining~\eqref{eq:lipschitz_result_for_proof_in_time} and~\eqref{eq:convergence_coefficients_for_proof_in_time}, \cref{theorem:approximation_convergence_sde} can be proved.
\begin{proof}[Proof of \cref{theorem:approximation_convergence_sde}]
    The error $e_d(t)\,=\,X(t)\,-\,X_d(t)$ satisfies
    \[
        e_d(t) \,=\, \int_0^{t} \vect F(X(\tau))\,-\,\vect F_d(X_d(\tau))\, d\tau \,+\, \int_0^{t} \vect A(X(\tau))\,-\,\vect A_d(X_d(\tau))\,\mathrm dW_\tau.
    \]
    Using the inequality $(a+b)^2\, {\leq} \,2a^2\,+\,2b^2$ and Cauchy-Schwarz, we have
    \begin{equation}
        \begin{aligned}
            \expect \left[ \sup_{0 \,\leq\, t \,\leq\, T} |e_d(t)|^2\right] \,& \leq\,  2\,T\,\expect\,\left[ \int_0^{T} |\vect F(X(\tau))\,-\,\vect F_d(X_d(\tau))|^2 \,\mathrm d\tau\right] \\
        &+\, 2\,\expect \left[\sup_{0 \leq t \leq T} \left| \int_0^{t} \vect A(X(\tau))\,-\,\vect A_d(X_d(\tau))\,\mathrm dW_\tau \right|^2\right].
        \end{aligned}
    \label{eq:main eq global}
    \end{equation}
    The first term in the right-hand side can be bounded by using the triangle inequality with the decomposition $\vect F(X(\tau))-\vect F_d(X_d(\tau)) = (\vect F(X(\tau)) - \vect F(X_d(\tau))) + (\vect F(X_d(\tau)) - \vect F_d(X_d(\tau)))$,
    the Lipschitz continuity of $\vect F(\cdot)$ and the convergence of $\vect F_d$ to $\vect F$:
    \begin{equation}
            \begin{aligned}
                & \expect \left[\int_0^{T} |\vect F(X(\tau))\,-\,\vect F_d(X_d(\tau))|^2\,\mathrm d\tau \right] \\
                & \qquad \qquad \leq \expect \left[2\,D(s)\,T\,d^{-s}\,+\,2\,C_L\,\int_0^{T}|X(\tau)\,-\,X_d(\tau)|^2\,d\tau\right] \\
                & \qquad \qquad \leq 2\,D(s)\,T\,d^{-s}\,+\,2\,C_L\,\int_0^{T} \expect\left[\sup_{0\, \leq \,t\, \leq \,\tau}|e_d(t)|^2\right]\,d\tau
            \end{aligned}
    \label{eq:1st term main eq}
    \end{equation}
    The second term can be bounded in a similar manner by using Burkholder–Davis–Gundy inequality, see for example~\cite[Theorem 3.28]{MR1121940}, and It\^o isometry :
    \begin{equation}
        \begin{aligned}
            &\expect \left[\sup_{0\, \leq \,t \leq \,T} \left|\int_0^{t} \vect A(X(\tau))\,-\,\vect A_d(X_d(\tau)) \,\mathrm dW_\tau \right|^2\right]  \\
            & \qquad \qquad \leq \left|\int_0^{T} \vect A\,(X(\tau))\,-\,\vect A_d(X_d(\tau))\, \mathrm dW_\tau\right|^2 \\
            & \qquad \qquad \leq \, 8\,D(s)\,T \,d^{-s}\,+\,8\,C_L\,\int_0^{T}\expect \left[\sup_{0\, \leq \,t\, \leq \,\tau}|e_d(t)|^2\right]\,d\tau.
        \end{aligned}
        \label{eq:2nd term main eq}
    \end{equation}
    Using \eqref{eq:1st term main eq} and \eqref{eq:2nd term main eq} in \eqref{eq:main eq global}, we obtain:
    \begin{equation*}
        \expect \left[ \sup_{0\, \leq \,t\, \leq \,T} |e_d(t)|^2\right]\, \leq \,
        4\,(T\,+\,4)\,\left(D(s)\,T \,d^{-s}\,+\,C_L\,\int_0^{T}\expect
            E\left[\sup_{0\, \leq \,t\, \leq \,\tau}|e_d(t)|^2\right]\,\mathrm d\tau\right).
    \end{equation*}
    By Gronwall's inequality, this implies:
    \begin{equation}
        \expect \left[ \sup_{0\, \leq \,t\, \leq \,T} |e_d(t)|^2\right]\, \leq \,4\,(T\,+\,4)\,D(s)\,T\,d^{-s}\,\exp \left(4\,(T\,+\,4)\,C_L\, T\right),
    \end{equation}
    which finishes the proof.
\end{proof}

\begin{remark}
Note that, as mentioned in~\cref{sec:main_results},
the convergence of the solution can still be proved if we only assume that the Lipschitz continuity and convergence of the coefficients hold locally,
provided there exists $p>2$ and a constant $K$ independent of $d$ such that the solutions of the equations
        $$
            \mathrm dX=\vect F(X)\,\mathrm dt\,+\,\vect A(X)\,\mathrm dW_t, \quad \quad X(0) \,=\,X_0,
        $$
        and
        $$
            \mathrm dX_d=\vect F_d(X_d)\,\mathrm dt\,+\,\vect A_d(X_d)\,\mathrm dW_t,  \quad \quad X_d(0) \,=\,X_0,
        $$
        satisfy the moment bounds
        $$
            \expect \left[\sup_{0 {\leq} t  {\leq} T} |X(t)|^p\right]\,{\vee}\,\expect
            \left[\sup_{0 {\leq} t  {\leq} T} |X_d(t)|^p\right] \, {\leq} \,K.
        $$
With these alternative assumptions, we can show that:
\begin{align*}
    \expect \left[\sup_{0\, \leq \,t\, \leq \,T} |X(t) \,-\,X_d(t)|^2 \right] \, &\leq  \,4\,(T\,+\,4)\,D_R(s)\,T\,d^{-s}\,\exp \left(4\,(T\,+\,4)\,C_R\, T\right) \\
    &+ 2\,K\, \left(\frac{2^p\,\delta}{p}\,+\,\frac{p-2}{R^p\,p\,\delta^{\frac{2}{p-2}}}\right).
\end{align*}
for any ${\delta}\,>\,0$ and $R\,>\,X_0$, and where $C_R$ and $D_R$ are the local constants for the assumptions.
The proof of this estimate is very similar to the one of the strong convergence of Euler-Maruyama scheme in \cite[Theorem 2.2]{higham2002strong}, and will thus not be repeated here.
From this estimate, we deduce that the solution of the approximate homogenized equation converges to the exact solution when $d\,\to\,{\infty}$.
\end{remark}

\section{Implementation of the Algorithm and Numerical Experiments}
\label{sec:numerics}
In this section,
we discuss the implementation of the algorithm and present some numerical experiments to
validate the method and illustrate our theoretical findings.
\subsection{Implementation details}
We discuss below the quadrature rules used and the approach taken for the
calculation of the matrix and right-hand side of the linear system of
equations~\eqref{eq:finite_dimensional_problem_in_matrix_form}.

The algorithm requires the calculation of a number of Gaussian integrals of the type:
\begin{equation}
    I = \int_{\real^n} f(y) \, \gaussian[\mu][\Sigma](y) \, dy.
\end{equation}
Several approaches,
either Monte Carlo-based or deterministic,
can be used for the calculation of such Gaussian integrals.
Probabilistic methods
offer an advantage when the dimension $n$ of the state space of the fast process is large,
but since the HMM is more efficient than our approach in that case,
in practice we don't use them.
Instead, we use a multi-dimensional quadrature rule obtained by tensorization of one-dimensional Gauss-Hermite quadrature rules.

For the calculation of the stiffness matrix,
we can take advantage of the diagonality of $A$ when the potential is equal to $V_{\mu,\Sigma} := \frac{1}{2} (y - \mu) \Sigma^{-1} (y-\mu) + \log(\sqrt{(2\pi)^n \det{\Sigma}})$.\footnote{
The constant $\log(\sqrt{(2\pi)^n \det{\Sigma}})$ in $V(\mu, \Sigma)$ is chosen so that $\int_{\real^n} e^{-V} \, dy = 1$.
}
Using the notation $\operator H_{\mu,\Sigma}$ to denote the same operator as in~\cref{lemma:bound_by_gaussian_norm},
and the shorthand notations $\hermite H_\alpha$ and $\hermitef h_\alpha$,
for $\alpha \in \nat^n$,
in place of $\hermite H_\alpha(y;\mu,\Sigma)$ and $\hermitef h_\alpha(y;\mu,\Sigma)$, respectively,
we have:
\begin{equation}
    A_{\alpha \beta} = -\int_{\real^n} \left( \operator H - \operator H_{\mu,\Sigma} \right) \hermitef h_\alpha \, \hermitef h_\beta \, dy
    - \int_{\real^n} \operator H_{\mu,\Sigma} \,\hermitef h_\alpha \, \hermitef h_\beta \, dy =: A^\delta_{\alpha \beta} + D_{\alpha \beta},
\end{equation}
where $D$ is a diagonal matrix whose entries can be computed explicitly and
\begin{equation}
    A^\delta_{\alpha \beta} = \int_{\real^n} \left( W - W_{\mu,\Sigma} \right) \hermite f_\alpha \hermite f_\beta \, dy = \int_{\real^n} (W - W_{\mu,\Sigma}) \, \gaussian[\mu][\Sigma] \, \hermite H_\alpha \hermite H_\beta \, dy,
\end{equation}
where $W_{\mu,\Sigma}$ is the potential obtained from $V_{\mu,\Sigma}$ according to \cref{eq:assumption_w}.
To simplify the calculation of these coefficients, we can expand the Hermite polynomials in terms of monomials:
\begin{equation}
    \hermite H_\alpha(y; \mu,\Sigma) = \sum_{|\beta| \leq d} c_{\alpha\beta} \, y^\beta.
\end{equation}
With this notation, we can write:
\begin{equation}
    A_{\alpha \beta}^\delta = \sum_{|\rho| \leq d} \, \sum_{|\sigma| \leq d} \, c_{\alpha \rho} \, c_{\beta \sigma} \int_{\real^n} (W - W_{\mu,\Sigma}) \, \gaussian[\mu][\Sigma] \,y^{\rho + \sigma} \, dy
                                =: \sum_{|\rho| \leq d} \, \sum_{|\sigma| \leq d} \, c_{\alpha \rho} \, c_{\beta \sigma} I_{\rho + \sigma},
\end{equation}
The integrals $I_{\alpha}$ are computed using a numerical quadrature.
Denoting by $w_i$ and $q_i$ the weights and nodes of the Gauss-Hermite quadrature, respectively,
$I_\alpha$ is approximated as
\begin{equation}
    I_{\alpha} \approx \sum_{i=1}^{N_q} w_i \left( W(q_i) - W_{\mu,\Sigma}(q_i) \right) \, \gaussian[\mu][\Sigma](q_i) \, q_i^{\alpha}, \quad \quad |\alpha| \leq 2d,
\end{equation}
where $N_q$ denotes the number of points in the quadrature.
Only the last factor of the previous expression depends on the multi-index $\alpha$,
so the numerical calculation of these integrals can be performed by evaluating for each grid point
    the value of $w_i \left( W(q_i) - W_{\mu,\Sigma}(q_i) \right) \, \gaussian[\mu][\Sigma](q_i)$ and
    the values of $q_i^\alpha$ for $|\alpha| \leq 2d$.

A similar method can be applied for the calculation of the right-hand side,
whose elements are expressed as:
\begin{equation}
    b_\alpha = \int_{\real^n} e^{-V/2} f \, e_\alpha \, dy.
\end{equation}
By expanding the Hermite functions in terms of Hermite polynomials multiplying $\gaussian[\mu][\Sigma]^{1/2}$,
the previous equation can be rewritten as
\begin{equation}
    b_\alpha = \sum_{|\beta| \leq d} c_{\alpha \beta} \int_{\real^n} \left(\frac{e^{-V}}{\gaussian[\mu][\Sigma]}\right)^{\frac{1}{2}} \, f(x,y) \, y^{\beta} \, \gaussian[\mu][\Sigma] \, dy,
\end{equation}
which is a Gaussian integral that can also be calculated using a multi-dimensional Gauss-Hermite quadrature.

\subsection{Numerical experiments}
Now we present the results of some numerical experiments.

The Euler-Maruyama scheme is used to approximate both $X(t)$ and $X_d(t)$ with a time step of $0.01$ for $T = 1$,
and $N_r = 50$ replicas of the driving Brownian motion are used for the numerical computation of expectations.
The $i$th replica of the discretized approximations of $X(t)$ and $X_d(t)$ are noted $X^{n,i}$ and $X_d^{n,i}$ respectively.
In most of the numerical experiments below, the error is measured by:
\begin{equation}
    \label{eq:error_measure_for_spectral_method}
    E(d) = \left(\frac{1}{N_r} \, \sum_{i=1}^{N_r}\max_{0 \leq n \, \Delta t \leq 1}\, |X^{n,i} \,-\,X_{d}^{n,i}|^2 \right)^{\frac{1}{2}},
\end{equation}
which is an approximation of the norm $||| \cdot |||$ used in \cref{theorem:approximation_convergence_sde}.

In the numerical experiments presented in this paper,
we have chosen the scaling parameter $\lambda$ in \cref{eq:choice_of_mu_and_sigma} by trial-and-error.
A natural extension of the work presented in this paper is to develop a systematic methodology for identifying the optimal scaling parameter,
see also the discussion in \cref{remark:scaling_parameter}.

\subsubsection{Test of the method for single well potentials}
\label{ssub:test_of_the_method_for_simple_problems}
For the two problems in this section,
the scaling parameter is chosen as \emph{$\lambda = 0.5$} for all degrees of approximation.
We start by considering the following problem.
\begin{equation}
    \left\{\begin{aligned}
            dx_{0t} &= -\frac{1}{\varepsilon}\,  \operator{L}\left[\cos{\left (x_{0t} + y_{0t} + y_{1t} \right )}\right]\,dt, \\
            dx_{1t} &= -\frac{1}{\varepsilon} \, \operator{L}\left[\sin{\left (x_{1t} \right )} \sin{\left (y_{0t} + y_{1t} \right )}\right]\,dt, \\
            dy_{0t} &= - \frac{1}{\varepsilon^2} \, \partial_{y_0} V(y)\, dt + \frac{1}{\varepsilon} \left[\cos{\left (x_{0t} \right )} \cos{\left (y_{0t} \right )} \cos{\left (y_{1t} \right )}\right] \, dt+ \frac{4}{\varepsilon} dW_{0t},\\
            dy_{1t} &= - \frac{1}{\varepsilon^2} \, \partial_{y_1} V(y) \, dt+ \frac{1}{\varepsilon} \left[\cos{\left (x_{0t} \right )} \cos{\left (y_{0t} + y_{1t} \right )}\right] \, dt+ \frac{4}{\varepsilon} dW_{1t},\\
        \end{aligned}\right.
    \label{eq:simple_2D_example}
\end{equation}
with
\begin{equation}
    \label{eq:potential_for_simple_2D_example}
    V(y) = y_{0}^{2} + y_{1}^{2} + 0.5 \left(y_{0}^{2} + y_{1}^{2}\right)^{2},
\end{equation}
and where $\operator{L} = -\grad V \cdot \nabla + \Delta$.
We have written the right-hand side of the equations for the slow processes $x_{0t}$ and $x_{1t}$ in this form
to ensure that the centering condition is satisfied.
The convergence of the approximate solution of the effective equation for this problem is illustrated in \cref{fig:simple_2D_example}.
Here the potential is very centered, so Hermite functions are well suited for the approximation of the solution, which is reflected in the very good convergence observed.
\begin{figure}
    \centering {%
        {\input{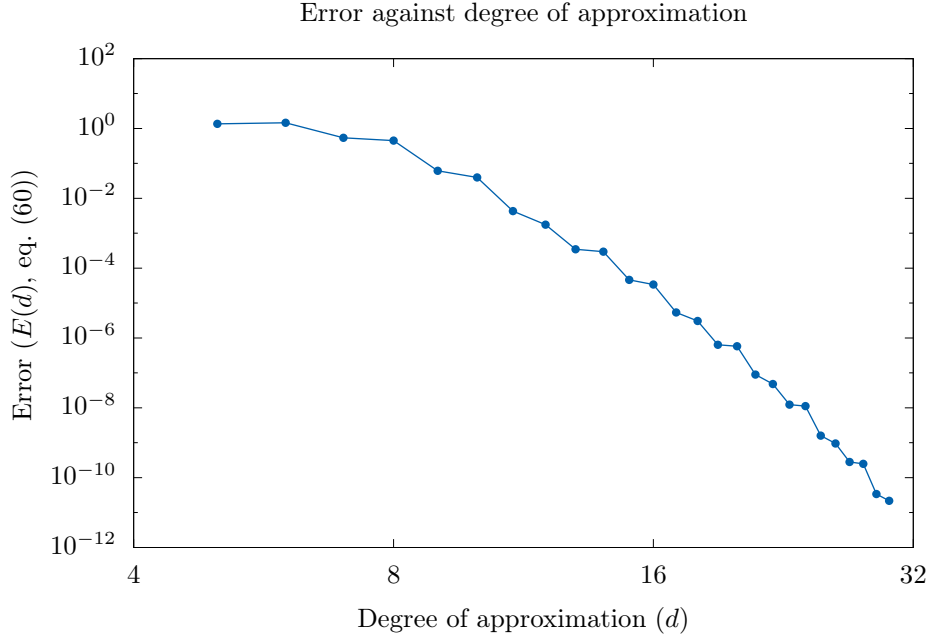}}
    }
    \caption{Error $E(d)$, see \cref{eq:error_measure_for_spectral_method}, for the fast-slow SDE \eqref{eq:simple_2D_example}.
    A super-algebraic convergence is observed.}
    \label{fig:simple_2D_example}
\end{figure}

In the next example, the state space of the fast process has dimension 3:
\begin{equation}
    \left\{\begin{aligned}
            dx_{0t} &=  -\frac{1}{\varepsilon}\,\operator{L}\left[\cos{\left (x_{0t} + y_{0t} + y_{1t} \right )}\right]\,dt,\\
            dx_{1t} &=  -\frac{1}{\varepsilon}\,\operator{L}\left[\sin{\left (x_{1t} \right )} \sin{\left (y_{0t} + y_{1t} + 2 y_{2t} \right )}\right]\,dt,\\
            dy_{0t} &=  -\frac{1}{\varepsilon^2}\,\partial_{y_{0}} V(y)\,dt + \frac{1}{\varepsilon}\left[\cos{\left (x_{0t} \right )} \cos{\left (y_{1t} \right )} \cos{\left (y_{0t} + y_{2t} \right )}\right]\, dt+\frac{\sqrt{2}}{\varepsilon}\,dW_{0t},\\
            dy_{1t} &=  -\frac{1}{\varepsilon^2}\,\partial_{y_{1}} V(y)\,dt + \frac{1}{\varepsilon}\left[\cos{\left (x_{0t} \right )} \cos{\left (y_{0t} + y_{1t} \right )}\right]\, dt+\frac{ \sqrt{2} }{\varepsilon}\,dW_{1t},\\
            dy_{2t} &=  -\frac{1}{\varepsilon^2}\,\partial_{y_{2}} V(y)\,dt +\frac{ \sqrt{2} }{\varepsilon}\,dW_{2t},\\
    \end{aligned}\right.
    \label{eq:simple_3D_example}
\end{equation}
with
\begin{equation}
    \label{eq:potential_for_simple_3D_example}
    V(y) = y_{0}^{4} + 2 y_{1}^{4} + 3 y_{2}^{4}.
\end{equation}
Because computing the effective coefficients is much more expensive computationally than in the previous case,
we measure the error for a given value of the slow variables, by
\begin{equation}
    \label{eq:error_measure_pointwise_for_spectral_method}
    e(d,x) = \frac{\abs{\vect F(x) - \vect F_d(x)}}{\abs{\vect F(x)}} + \frac{\abs{\vect A(x) - \vect A_d(x)}}{\abs{\vect A(x)}}.
\end{equation}
The value we chose for the comparison is $x = (0.2, 0.2)$, for which the denominators in the previous equation are non-zero.
The relative error on the homogenized coefficients is illustrated in \cref{fig:simple_3D_example}.
In this case, the method also performs very well,
although it is slightly less accurate than in the previous example.

\begin{figure}
    \centering {%
        {\input{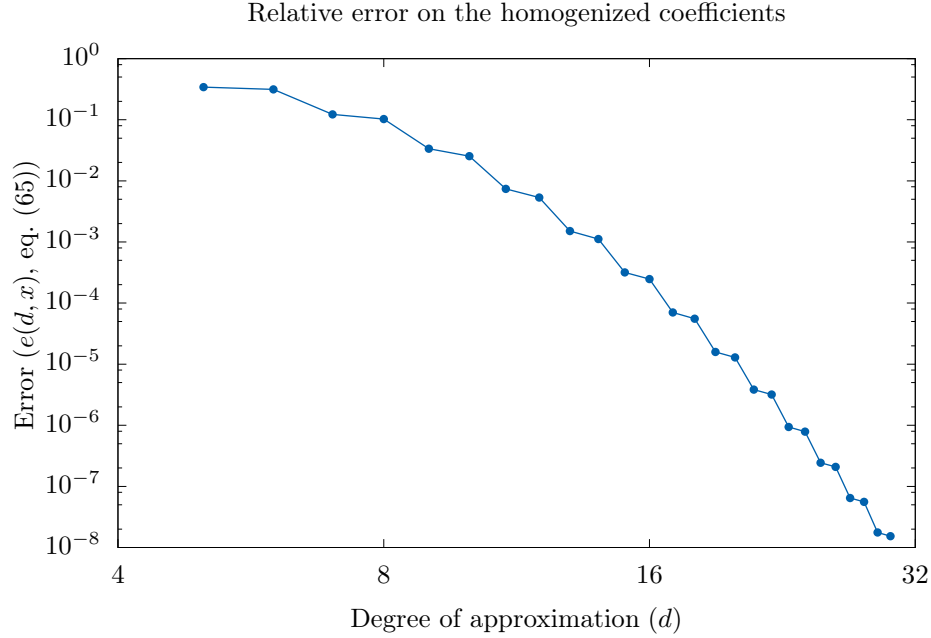}}
    }
    \caption{Relative error of the homogenized coefficients, $e(d,x)$, see \cref{eq:error_measure_pointwise_for_spectral_method},
        for the fast/slow SDE \eqref{eq:simple_3D_example} at $x = (0.2, 0.2)$.
    In this case, the convergence is also super-algebraic.}
    \label{fig:simple_3D_example}
\end{figure}

\subsubsection{Test of the method for potentials with multiple wells}
Now we consider multiple-well potentials that lead to multi-modal distributions.
The first potential that we analyze is the standard bistable potential,
\begin{equation}
    \label{eq:potential_for_1D_bistable_system}
    V(y)  = y^4/4 - y^2/2.
\end{equation}
We consider the fast/slow SDE system:
\begin{equation}
    \left\{
        \begin{aligned}
            dx_t &= -\frac{1}{\varepsilon}\operator{L}\left(x_t \, \sin(y_t)\right)\,dt, \\
            dy_t &= -\frac{1}{\varepsilon^2} \partial_y{V}(y_t)\,dt + \frac{\sqrt{2}}{\varepsilon}\, dW_t.
        \end{aligned}
    \right.
\label{eq:1D_bistable_system}
\end{equation}
We choose the parameter $\lambda$ in \cref{eq:choice_of_mu_and_sigma} to be $\lambda = 0.5$.
The convergence of the method is illustrated in \cref{fig:1D_bistable_system}.
Although the method is less accurate than in the previous cases, a super-algebraic convergence can still be observed,
and a very good accuracy can be reached by choosing a high enough value for the degree of approximation.
Note that the computational cost in this case is very low---the numerical solution can be calculated in a matter of seconds on a personal computer.
\begin{figure}[!ht]
    \centering {%
        {\input{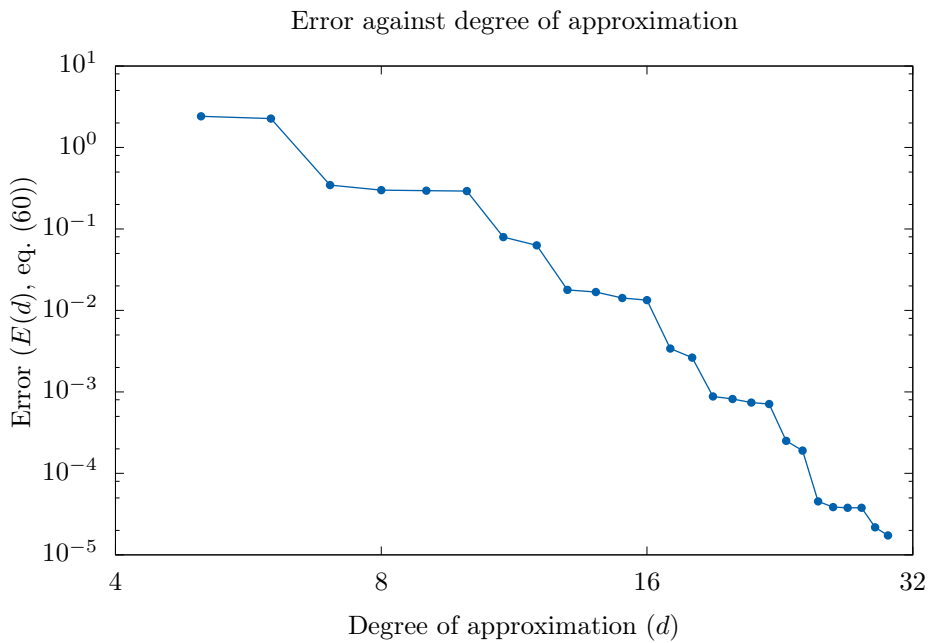}}
    }%
    \caption{Error $E(d)$, see \cref{eq:error_measure_for_spectral_method}, for the fast/slow SDE~\eqref{eq:1D_bistable_system}.}
    \label{fig:1D_bistable_system}
\end{figure}

Next we consider the tilted bistable potential
\begin{equation}
    \label{eq:potential_for_1D_swarming}
    V(y) = y^4/4 - y^2/2 + 10y,
\end{equation}
which corresponds to the case $\gamma = 1$, $\delta = 10$ in the examples considered in \cite{goudonefficient}, and the fast/slow SDE
\begin{equation}
    \left\{
        \begin{aligned}
            dx_t &= -\frac{1}{\varepsilon}\operator{L}\left(x_t \, \sin(y_t) + y_t^2 \right)\,dt, \\
            dy_t &= -\frac{1}{\varepsilon^2} \partial_y{V}(x_t,y_t)\,dt + \frac{\sqrt{2}}{\varepsilon}\, dW_t.
        \end{aligned}
    \right.
\label{eq:1D_swarming}
\end{equation}
The convergence of the solution in this case is presented in \cref{fig:1D_swarming},
for the scaling parameter $\lambda = 1$.
Due to the presence of a strong linear term,
the potential is actually very localized, see \cref{fig:tilted_potential},
which results in good convergence of the spectral method.
\begin{figure}[!ht]
    \centering
    \centering {%
        {\input{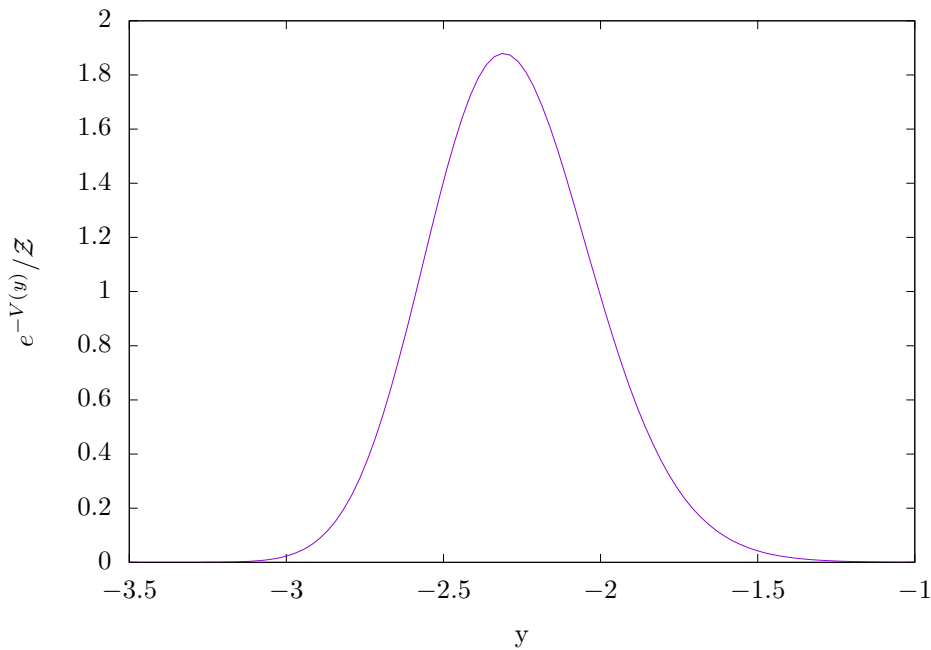}}
    }%
    \caption{Probability density $e^{-V(\cdot)}/\mathcal Z$ associated to the potential~\eqref{eq:potential_for_1D_swarming}.}
    \label{fig:tilted_potential}
\end{figure}
\begin{figure}[!ht]
    \centering {%
        {\input{picture_1D_swarming_time_error.tex}}
    }%
    \caption{Error $E(d)$, see \cref{eq:error_measure_for_spectral_method}, for the fast/slow system~\eqref{eq:1D_swarming}.}
    \label{fig:1D_swarming}
\end{figure}

Finally, we consider a three-well potential in $\real^2$,
\begin{equation}
    \label{eq:three_well_potential}
    {V(y) = \left(\left(y_{0} - 1\right)^{2} + y_{1}^{2}\right) \left(\left(y_{0} + \frac{1}{2}\right)^{2} + \left(y_{1} - \frac{\sqrt{3}}{2}\right)^{2}\right) \left(\left(y_{0} + \frac{1}{2}\right)^{2} + \left(y_{1} + \frac{\sqrt{3}}{2}\right)^{2}\right),}
\end{equation}
and the following fast/slow SDE:
\begin{equation}
    \left\{\begin{aligned}
            dx_{0t} &=  -\frac{1}{\varepsilon}\,\operator{L}\left[\cos{\left (x_{0t} + y_{0t} + y_{1t} \right )}\right]\,dt,\\
            dx_{1t} &=  -\frac{1}{\varepsilon}\,\operator{L}\left[\sin{\left (x_{1t} \right )} \sin{\left (y_{0t} + y_{1t} \right )}\right]\,dt,\\
            dy_{0t} &=  -\frac{1}{\varepsilon^2}\,\partial_{ y_{0} } V(y)\,dt + \frac{1}{\varepsilon}\left[\cos{\left (x_{0t} \right )} \cos{\left (y_{0t} \right )} \cos{\left (y_{1t} \right )}\right]\,dt+\frac{ \sqrt{2} }{\varepsilon}\,dW_{0t},\\
            dy_{1t} &=  -\frac{1}{\varepsilon^2}\,\partial_{ y_{1} } V(y)\,dt + \frac{1}{\varepsilon}\left[\cos{\left (x_{0t} \right )} \cos{\left (y_{0t} + y_{1t} \right )}\right]\,dt+\frac{ \sqrt{2} }{\varepsilon}\,dW_{1t}.\\
    \end{aligned}\right.
    \label{eq:potential_with_three_wells_system}
\end{equation}
For this fast/slow SDE, we choose $\lambda = 0.35$.
A contour plot of the potential is shown in \cref{fig:contour_potential_with_three_wells}, and the convergence graph is presented in \cref{fig:potential_with_three_wells_system}.
In this case the error is very large for degrees of approximation lower than 10,
beyond which the convergence is clear and super-algebraic.
The accuracy reached with a degree of approximation equal to 30 is of the order of $1 \times 10^{-4}$,
which is good in comparison with the accuracy that can be achieved using Monte Carlo-based methods.
\begin{figure}[!ht]
    \centering {%
        {\input{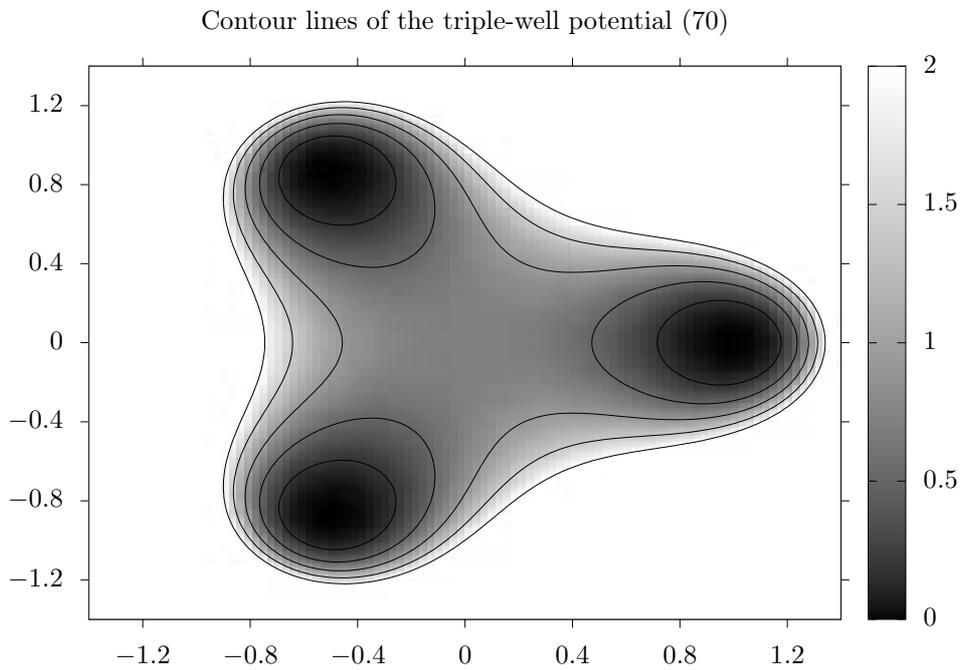}}
    }
    \caption{Potential~\eqref{eq:three_well_potential}, used in equation~\eqref{eq:potential_with_three_wells_system}.}
    \label{fig:contour_potential_with_three_wells}
\end{figure}
\begin{figure}[!ht]
    \centering {%
        {\input{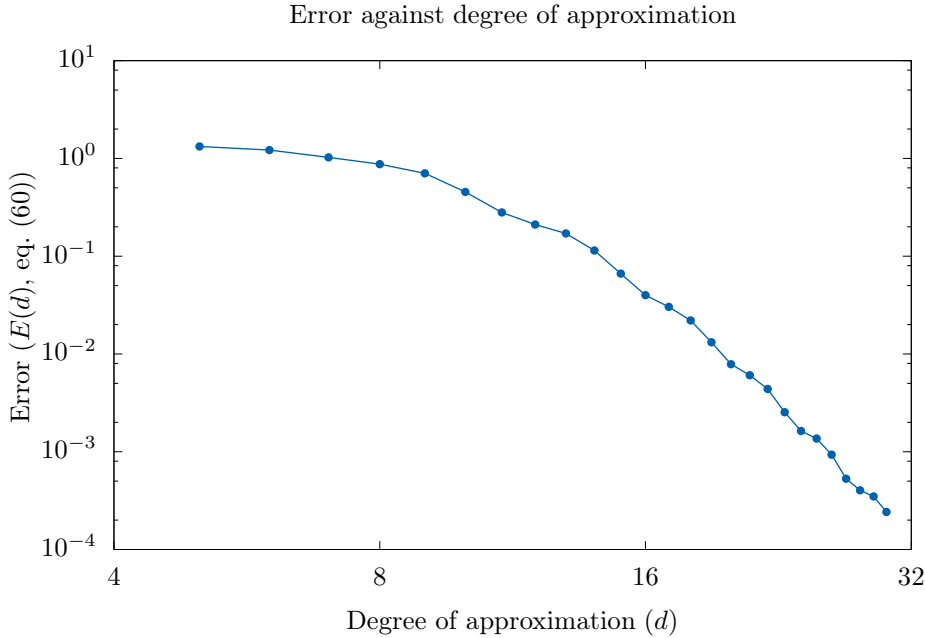}}
    }
    \caption{Error $E(d)$, see \cref{eq:error_measure_for_spectral_method},
    for the fast/slow system~\eqref{eq:potential_with_three_wells_system}.}
    \label{fig:potential_with_three_wells_system}
\end{figure}

\subsubsection{Discretization of a multiscale stochastic PDE}
\label{ssub:discretization_of_a_multiscale_spde}
As mentioned in the introduction, our numerical method is particularly well-suited for
the solution of singularly perturbed stochastic PDEs (SPDEs),
and constitutes a very good complement to the method proposed in \cite{abdulle2012numerical}.
Let us recall how the method introduced in~\cite{abdulle2012numerical} works for a singularly perturbed SPDE of the following form
\begin{equation}
    \frac{{\partial}u}{{\partial}t}\,=\, \frac 1{{\varepsilon}^2} \mathcal A u + \frac 1{\varepsilon} F(u) +\frac1{\varepsilon} Q \dot W,
    \label{eq:spde}
\end{equation}
posed in a bounded domain of $\real^m$ with suitable boundary conditions.
In~\cref{eq:spde}, $\mathcal A$ is a differential operator, assumed to be nonpositive and selfadjoint in a Hilbert space $\mathcal H$, and with compact resolvent.
It is furthermore assumed that $\mathcal A$ has a finite dimensional kernel, denoted by $\mathcal M$.
The term $W$ denotes a cylindrical Wiener process on $\mathcal H$ and
$Q$ denotes the covariance operator of the noise.
It is assumed that $Q$ and $\mathcal A$ commute, and that the noise acts only on the orthogonal complement of $\mathcal M$, denoted by $\mathcal M^{\perp}$.
The function $F({\cdot})$ is a polynomial function representing a nonlinearity that has to be such that the above scaling makes sense.\footnote{
    i.e., the centering condition is satisfied.
}

Since $\mathcal A$ is selfadjoint with compact resolvent, there exists an orthonormal basis of $\mathcal H$ consisting of eigenfunctions of $\mathcal A$.
We denote by $\{\lambda_k, e_k\}$ the eigenvalues and corresponding eigenfunctions of $\mathcal A$.
We arrange the eigenpairs by increasing absolute value of the eigenvalues,
so the $m$ first eigenfunctions are in the kernel of the differential operator, $\mathcal M = \Span\{e_1, \dots, e_m\}$.
Formally, the cylindrical Brownian motion can be expanded in the basis as $W(t) \,=\, \sum^{\infty}_{i=1} \, e_i \, w_i(t)$, where $\left\{w_i\right\}_{i=1}^{\infty}$ are independent Brownian motions.
The assumption that the covariance operator $Q$ commutes with the differential operator $\mathcal A$ means that this operator satisfies $Q\,e_i \,=\, q_i\, e_i$, while the assumption that the noise only acts on \,$\mathcal M^{\perp}$ implies that $q_i \,=\, 0$ for $i\,=\,1,\,2,\,{\dots}\, , \, m$.

We now summarize how the dynamics of the slow modes in~\eqref{eq:spde} can
be approximated by solving a multiscale system of SDEs using the methodology developed in~\cite{abdulle2012numerical}.

First, we write the solution of~\eqref{eq:spde} as
\begin{equation*}
    u = x+y, \,\,\,\text{with}\,\,\,x = \sum_{k=1}^{m} \,x_{k}\,e_{k}
    \,\,\,\text{and}\,\,\, y = \sum_{k=m+1}^{\infty} \,y_{k}\,e_{k}.
\end{equation*}
Note that $x = \mathcal Pu$, and $y = (I-\mathcal P)u$, where $\mathcal P$ is
the projection operator from $\mathcal H$ onto $\mathcal M$.
By assumption, the noise term can be expanded in the same way, as $\sum_{k=1}^{\infty}\,q_{k}\,e_{k}\,\dot w_{k}(t)$.
Substitution of these expansions in the SPDE gives:
\begin{equation*}
        \frac d{dt}  \left(\sum_{k=1}^{m}\,x_{k}\,e_{k}\,\,+\sum_{ k=m+1}^{\infty}
            y_{k}\,e_{k} \right) =  -\frac 1{{\varepsilon}^{2}}\sum_{ k=m+1}^{\infty}
        {\lambda}_{k}\,y_{k}\,e_{k} + \frac 1 {\varepsilon}F(u) + \frac 1 {\varepsilon}\,
        \sum_{k=m+1}^{\infty}q_{k}\,e_{k}\,{\dot w}_{k}(t).
\end{equation*}
The equations that govern the evolution of the coefficients $x_k$ and $y_k$
can be  obtained by taking the inner product (of $\mathcal H$) of both sides of the above equation by each
of the eigenfunctions of  the operator, and using orthonormality :
\begin{equation}
    \label{eq:infinite_system_of_sdes}
    \left\{\begin{aligned}
            \dot x_{i} &= \frac 1 {\varepsilon}\langle F(u), e_{i}\rangle & \quad i & = 1,{\dots},
            m;\\ \dot y_{i} &= -\frac 1 {{\varepsilon}^{2}} {\lambda}_{i}\,y_{i} + \frac 1{\varepsilon}\langle
            F(u), e_{i}\rangle +\frac 1{{\varepsilon}}q_{i}\,{\dot w}_{i} & \quad i &= m+1,m+2, {\dots}
        \end{aligned} \right.
\end{equation}
\Cref{eq:infinite_system_of_sdes} can be written in the form
\begin{equation}
    \left\{\begin{aligned}
            \dot x &= \frac 1 {\varepsilon}a(x,y),\\ \dot y &= \frac 1 {{\varepsilon}^{2}} \mathcal A\,y
            + \frac 1 {\varepsilon}b(x,y) +\frac 1{{\varepsilon}}Q\,{\dot W},
        \end{aligned} \right.
    \label{eq:spde_as_fast_slow_system_of_spde}
\end{equation}
where $a(x,y)$ and $b(x,y)$ are the projections
of $F(u)$ on $\mathcal M$ and $\mathcal M^{\perp}$, respectively:
\begin{equation*}
    a(x,y) = \sum_{i=1}^{m}a^{i}(x,y)\,e_{i}\,\,\,\,\text{with}\,\,\,\,
    a^{i}(x,y) = \langle F(x+y), e_{i}\rangle,
\end{equation*}
and
\begin{equation*}
    b(x,y) = \sum_{i=m+1}^{\infty}b^{i}(x,y)\,e_{i}\,\,\,\,\text{with}\,\,\,\,
    b^{i}(x,y) = \langle F(x+y), e_{i}\rangle.
\end{equation*}
The scale separation now appears clearly.
We now truncate the fast process in \cref{eq:spde_as_fast_slow_system_of_spde} as $y\,{\approx}\,\sum_{i=m+1}^{m+n}\,y_{i}\,e_{i}$ to
derive the following finite dimensional system is obtained:
\begin{equation}
    \left\{\begin{aligned}
            \dot x_{i}\,&=\,\frac 1 {\varepsilon}a^{i}(x,y) & \quad i = 1,{\dots}, m;  \\ \dot
            y_{i}\,&=-\frac 1 {{\varepsilon}^{2}} {\lambda}_{i} y_{i}\,+\,\frac
            1{\varepsilon}\,b^{i}(x,y)\,+\,\frac 1{{\varepsilon}}q_{i}\,{\dot w}_{i} & \quad i = m+1, {\dots}m+n,
        \end{aligned} \right.
    \label{eq:system_sdes_from_spde}
\end{equation}
In~\cite{abdulle2008effectiveness}, the authors investigate
the use of the heterogeneous multiscale method (HMM)
for solving the problem~\eqref{eq:system_sdes_from_spde},
and show that a good approximation can be obtained using this method.
However, when the nonlinearity is a polynomial function of $u$, the function $a$ in the system above,
which also appears on the right-hand side of the Poisson equation, is polynomial in $x$ and $y$.
In addition, the generator of this system of stochastic differential equations is of Ornstein-Uhlenbeck type to leading order, and so its eigenfunctions are Hermite polynomials.
This means that the right-hand side can be expanded exactly in Hermite polynomials, and so the exact effective coefficients can be computed.
Note that although equivalent, applying the unitary transformation is not necessary in this case,
as we can work directly with Hermite polynomials in the appropriate weighted $L^2$ space.

We consider the SPDE~\eqref{eq:spde},
with $\mathcal A = \frac{\partial^2}{\partial x^2} + 1$ and $F(i) = u^2 \frac{\partial u^2}{\partial x}$,
posed on $[-{\pi},\,{\pi}]$ with periodic boundary conditions:
\begin{equation}
    \label{eq:example_spde}
    \frac{\partial u}{\partial t} \,=\,\frac{1}{\varepsilon^2}\, \left( \frac{\partial^2}{ \partial x^2} \,+\,1\right)\,u \,+\, \frac{1}{\varepsilon} \, u^2\, \frac{\partial u^2}{\partial x} \,+\, \frac{1}{\varepsilon} \, Q \dot W.
\end{equation}
The eigenfunctions of $\operator{A}$ on $[-\pi, \pi]$ with periodic B.C. are
\begin{equation*}
    e_i \,=\, \left\{
        \begin{aligned}
            & \frac{1}{\sqrt{\pi}}\sin\left(\frac{i+1}{2}\,x\right) &\quad \text{ if $i$ is odd,} \\
            & \frac{1}{\sqrt {\pi}}\cos\left(\frac{i}{2}\,x\right) &\quad \text{ if $i$ is even,}
        \end{aligned} \right .
\end{equation*}
and the corresponding eigenvalues are
$\lambda_i = 1 - \frac{(i+1)^2}{4}$ if $i$ is odd and $\lambda_i = 1 - \frac{i^2}{4}$ if $i$ is even.
In this case the null space of $\operator{A}$ is two-dimensional.
We consider a noise process of the form:
\begin{equation}
    \label{eq:noise_process_for_equation}
    Q \dot W = \sum_{i=3}^{\infty} q_i \, \dot w_i.
\end{equation}
Following the methodology outlined above, we approximate the solution by a truncated Fourier series:
\begin{equation}
    \label{eq:fourier_expansion_example_spde}
    u \,=\, x_{1}\,e_1 \,+\, x_2\,e_2 \,+\, \sum_{ i\,=\,3}^{n\,+\,2} y_i \, e_i.
\end{equation}
Substituting in the nonlinearity and taking the inner product with each of the eigenfunctions, a system of equation of the type \eqref{eq:system_sdes_from_spde} is obtained.
The operator $\operator{A}$ and the nonlinearity were chosen so that
the centering condition is satisfied.
The homogenized equation for the slow variables $(x_1, x_2)$ reads
\begin{equation}
    \label{eq:spde_effective_equation}
    d X_t \,=\, \vect F(X_t) \,dt + \vect A(X_t) \, dW_t,
\end{equation}
where $\vect F(\cdot)$  and $\vect A(\cdot)$ are given by equations~\eqref{eq:effective_drift} and~\eqref{eq:effective_diffusion}, respectively,
and $W$ is a standard Wiener process in $\real^2$.
The Euler-Maruyama solver was used for both the macro and micro solvers, and the parameters of the HMM were chosen as
\begin{equation}
    \label{eq:parameters_hmm}
    ({\delta}t/{\varepsilon}^2, N_T, M,N, N') = ( 2^{-p}, 16, 1 , 10{\times}2^{3p}, 2^pp).
\end{equation}
Here $\delta t$ is the time step of the micro-solver,
$N_T$ is the number of steps that are omitted in the time-averaging process to reduce transient effects,
$M$ is the number of samples used for ensemble averages,
and $N$, $N'$ are the number of time steps employed for
the calculation of time averages and
the discretization of integrals originating from Feynman-Kac representation formula~\eqref{eq:feynman-kac_formula}, respectively.
See~\cite{weinan2005analysis, vanden2003fast} for a more detailed description of the method and
a detailed explanation of the parameters in~\eqref{eq:parameters_hmm}.
In \cref{fig: spde different ps x1,fig: spde different ps x2},
we compare the solutions obtained using the HMM method with the one obtained using our approach,
using the same macro-solver and the same replica of the driving Brownian motion for both,
and with the initial condition $x_{i0} = 1.2$ for $i = 1, \dots, m$.
The former is denoted by $\hat X^n$ and the latter by $X^n$.
Notice that when the value of the parameter $p$ increases,
the solution obtained using the HMM converges to the exact solution obtained using the Hermite spectral method.

We now investigate the dependence on the precision parameter $p$ of the error between the homogenized coefficients.
The same error measure as in~\cite{weinan2005analysis} is used to compare the two methods:
\begin{equation}
    \label{eq:error_measure_for_comparison_with_hmm}
    E_p \,=\, \frac{{\Delta}t}{T}~ \left(\sum_{n \,{\leq}\,T/{\Delta}t}| \vect F^p_{HMM}(X^n) - \vect F_{Sp}(X^n)| \,+\, |\vect A_{HMM}^p(X^n) \,-\, \vect A_{Sp}(X^n)| \right).
\end{equation}
Here $\vect F^p_{HMM}$ and $\vect A_{HMM}^p$ are the drift and diffusion coefficients obtained using
the HMM with the precision parameter equal to $p$,
while $\vect F_{Sp}$ and $\vect A_{Sp}$ are the coefficients given by the Hermite spectral method developed in this paper.
Given the choice of parameters~\eqref{eq:parameters_hmm},
the theory developed in~\cite{weinan2005analysis} predicts that
the error should decrease as $\mathcal O(2^{- p })$.
This error is presented in \cref{fig: spde error} as a function of the precision parameter $p$,
showing a good agreement with the theory developed in \cite{abdulle2012numerical,weinan2005analysis}.

For the SPDE described above
our method based on the solution of the Poisson equation associated with~\eqref{eq:system_sdes_from_spde} using Hermite polynomials
does recover exactly the corresponding effective parameters,
and the only source of error is the macroscopic discretization scheme.
This is in sharp contrast with the HMM-based method developed in \cite{abdulle2012heterogeneous},
for which the micro-averaging process to recover the effective coefficients represents a non-negligible computational cost.
\begin{figure}[!ht]
    \begin{center}
    {
        \resizebox{1.0\textwidth}{!}{\input{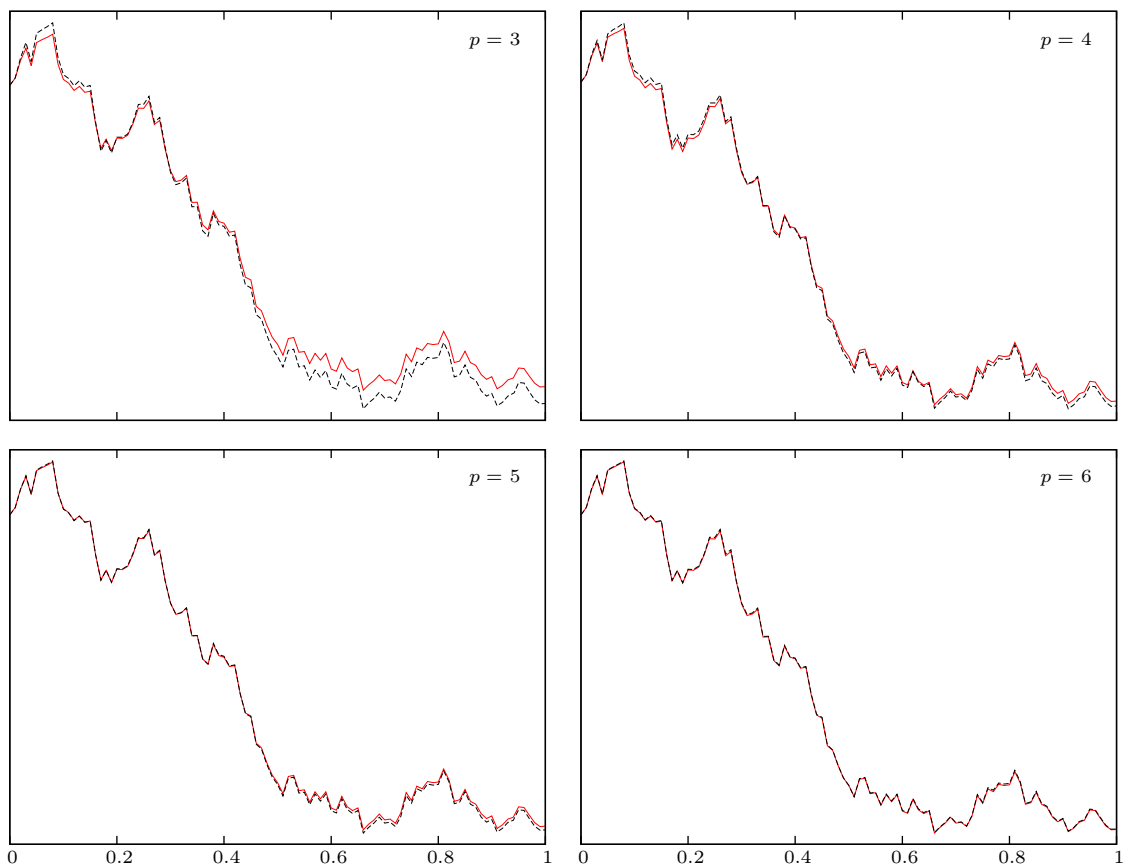}}
    }
    \end{center}
    \caption{Evolution of the coefficient $x_{1}$ of the first term in the Fourier expansion~\eqref{eq:fourier_expansion_example_spde}
        of the the solution to the SPDE~\eqref{eq:example_spde},
        obtained numerically by the HMM (black) and the Hermite spectral method (red),
        for one sample of the driving Brownian motion.
    }
    \label{fig: spde different ps x1}
\end{figure}
\begin{figure}[!ht]
    \begin{center}
    {
        \resizebox{1.0\textwidth}{!}{\input{picture_spde_comparison_x2.tex}}
    }
    \end{center}
    \caption{Evolution of the coefficient $x_2$ of the second term in the Fourier expansion~\eqref{eq:fourier_expansion_example_spde}
        of the the solution to the SPDE~\eqref{eq:example_spde},
        obtained numerically by the HMM (black) and the Hermite spectral method (red),
        for one sample of the driving Brownian motion.
    }
    \label{fig: spde different ps x2}
\end{figure}

\begin{figure}[!ht]
    \begin{center}
    {
        \resizebox{1.0\textwidth}{!}{\input{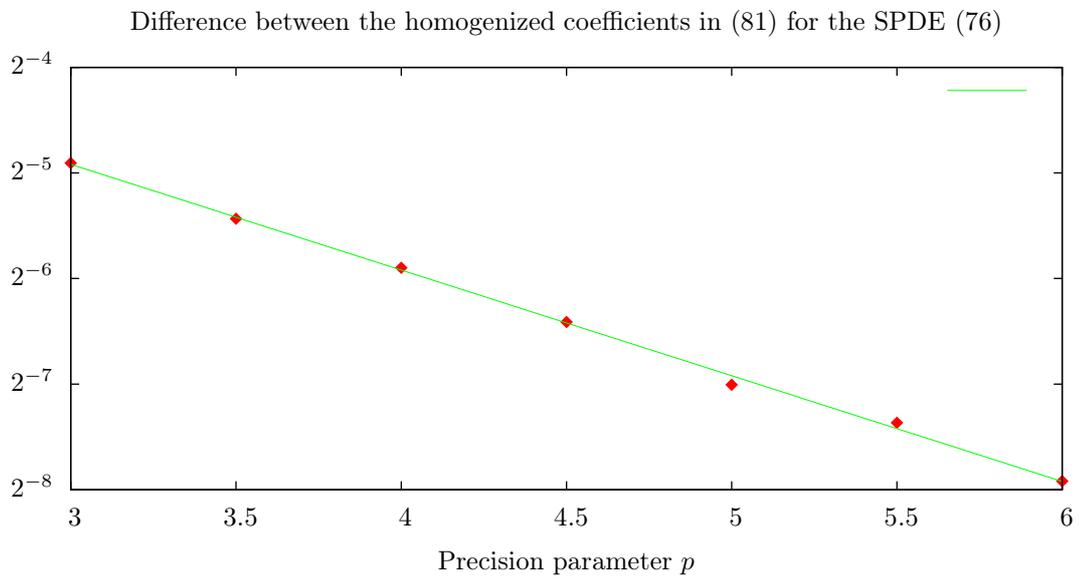}}
    }
    \end{center}
    \caption{Error between the homogenized coefficients (see \cref{eq:error_measure_for_comparison_with_hmm}) for the SPDE~\eqref{eq:example_spde},
        as a function of the precision parameter $p$.
        The green line, obtained by polynomial fitting, has slope $-1.01$ in the $p-\log_2(E_p)$ plane,
        which is close to the theoretical value of -1, showing a perfect agreement with the theory.
    }
    \label{fig: spde error}
\end{figure}

\clearpage
\section{Conclusion and Further Work}
\label{sec:conclusion_and_perspectives_for_future_work}

In this paper, we proposed a new approach for the numerical approximation of
the slow dynamics of fast/slow SDEs for which a homogenized equation exists.
Starting from the appropriate Poisson equation, the same unitary transformation as
in \cite{goudonefficient} was utilized to obtain formulas for
the drift and diffusion coefficients
in terms of the solution to a Schr\"odinger equation.
This equation is solved at each discrete time by means of a
spectral method using Hermite functions, from which approximations of the
homogenized drift and diffusion coefficients were calculated. A stochastic
integrator was then used to evolve the slow variables by one time step, and the
procedure is repeated.

Building on the work of \cite{gagelman2012spectral}, spectral convergence of
the homogenized coefficients was rigorously established, from which weak
convergence of the discrete approximation in time to the exact homogenized
solution was derived. In the final section, the accuracy and efficiency of the
proposed methodology were examined through numerical experiments.

The method presented, although not as general as the HMM, has proven more
precise and more efficient for a broad class of problems. It performs
particularly well for singularly perturbed SPDEs,
and constitutes in this case a good complement to the HMM-based method
presented in~\cite{abdulle2012numerical}. It also works comparatively very well
when the fast dynamics is of relatively low dimension---typically less than or
equal to 3---and especially so when the potential is localized, since fewer
Hermite functions are required to accurately resolve the Poisson equations in
this situation. Our method also has several advantages compared to the approach
taken in~\cite{goudonefficient}: it does not require truncation of the domain,
does not require the calculation of the eigenvalues and eigenfunctions of the Schr\"odinger operator,
and has better asymptotic convergence properties.

The limitations of the method are two-fold;
its generality is limited by the requirement of the gradient structure for fast dynamics,
and its efficiency is limited by the curse of dimensionality,
which causes the computation time to become prohibitive when the dimension of the state space of the fast process increases.

The extent to which some of these constraints can be lifted constitutes an
interesting topic for future work. We believe that it is possible to generalize
our method to a broader class of problems while retaining its efficiency and
accuracy. In addition, high-dimensional integrals could be computed more
efficiently. For example, an alternative to the
tensorized quadrature approach taken in this work is to use a sparse grid
method; such a method can in principle offer the same degree of polynomial exactness with a
significantly lower number of nodes, see e.g.~\cite{MR1669959,kaarnioja2013smolyak}.

\paragraph{Acknowledgments}
The authors thank Andrew Duncan, Gabriel Stoltz and Julien Roussel for useful discussions.
A. Abdulle is supported by the Swiss national foundation.
G.A. Pavliotis acknowledges financial support by the Engineering and Physical Sciences Research Council of the UK through Grants Nos. EP/L020564, EP/L024926 and EP/L025159.
U. Vaes is supported through a Roth PhD studentship by the Department of Mathematics, Imperial College London.

\appendix
\section{Weighted Sobolev Spaces}
\label{sec:weighted_sobolev_spaces}
In this section, we recall a few results about weighted Sobolev spaces
that are needed for the analysis presented in \cref{sec:analysis_of_the_method}.
For more details on this topic, see~\cite{gagelman2012spectral, MR2656512, MR3443169, lorenzi2006analytical}.
Throughout the appendix, $V$ denotes a smooth confining potential,
whose derivatives are all bounded above by a polynomial, and such that $\rho := e^{-V}$ is normalized.
\begin{definition}
    \label{definition:weighted_l2_space}
    The weighted $L^2$ space $\lp{2}{\real^n}[\rho]$ is defined as
    \[
        \lp{2}{\real^n}[\rho] = \left\{u \text{ measurable }: \int_{\real^n} u^2 \,\rho \, dy < \infty\right\}.
    \]
    It is a Hilbert space for the inner product given by:
    \[
        \ip{u}{v}[\rho] = \int_{\real^n} u \, v \, \rho \, dy.
    \]
\end{definition}
\begin{definition}
    \label{definition:weighted_sobolev_spaces_probability_density}
    The weighted Sobolev spaces $\sobolev{s}{\real^n}[\rho]$, with $s \in \nat$, is defined as
    \[
        \sobolev{s}{\real^n}[\rho] = \left\{u \in \lp{2}{\real^n}[\rho]: \partial^\alpha u \in \lp{2}{\real^n}[\rho] \, \forall \abs{\alpha} \leq s\right\}.
    \]
    It is a Hilbert space for the inner product given by:
    \[
        \ip{u}{v}[s][\rho] = \sum_{\abs{\alpha} \leq s} \ip{\partial^\alpha u}{\partial^\alpha v}_\rho
    \]
\end{definition}
We also define the following spaces.
\begin{definition}
    \label{definition:weighted_sobolev_spaces_operator}
    Given $s \in \nat$ and a nonnegative selfadjoint operator $-\operator{L}$ on a Hilbert space $H$ of functions on $\real^n$,
    we define $\sobolev{s}{\real^n}[\operator{L}]$ as the space obtained by completion of $\test{\real^n}$ for the inner product:
    \[
        \ip{u}{v}[s][\operator{L}] = \sum_{i=0}^{s} \ip{(-\operator{L})^iu}{v}_{H}.
    \]
    The associated norm will be denoted by $\norm{\cdot}[s][\operator{L}]$.
\end{definition}
It can be shown that $\test{\real^n}$ is dense in $\sobolev{1}{\real^n}[\rho]$, see~\cite{villani2006hypocoercivity}.
By integration by parts, this implies that $\sobolev{1}{\real^n}[\rho] = \sobolev{1}{\real^n}[\operator{L}]$,
where $-\operator{L}$ is the nonnegative selfadjoint operator on $\lp{2}{\real^n}[\rho]$ defined by $\operator{L} = \Delta - \grad{V}\cdot \grad$.
We now make the additional assumption that the potential $V$ satisfies
\begin{equation}
    \lim_{\abs{y} \to \infty} \left(\frac{1}{4} \abs{\grad V}^2 - \frac{1}{2}\Delta V\right) = \infty \quad \text{and} \quad \lim_{\abs{y} \to \infty}\abs{\grad V} = \infty.
    \label{eq:spectral_gap_assumption}
\end{equation}
With this, the following compactness result holds.
\begin{proposition}
    Assume that~\eqref{eq:spectral_gap_assumption} holds.
    Then the embedding $\sobolev{1}{\real^n}[\rho] \subset \lp{2}{\real^n}[\rho]$ is compact,
    and the measure $\rho$ satisfies Poincaré inequality:
    \[
        \int_{\real^n} {(u - \bar u)}^2\, \rho \, dy \leq C \int_{\real^n} \abs{\grad{u}}^2 \,\rho\, dy \quad \forall u \in \sobolev{1}{\real^n}[\rho],
    \]
    where $\bar u = \int_{\real^n} u\, \rho\, dy$.
\label{proposition:poincare_inequality_weighted_measure}
\end{proposition}
\begin{proof}
    See \cite{lorenzi2006analytical}, sec. 8.5, p. 216.
\end{proof}
\begin{remark}
Alternative conditions on the potential that ensure that the corresponding Gibbs measure satisfies a Poincar\'{e} inequality are presented in~\cite[Theorem 2.5]{LelievreStoltz2016}.
\end{remark}

Now we consider the unitary transformation $e^{-V/2}: \lp{2}{\real^n}[\rho] \to \lp{2}{\real^n}$, and characterize the spaces obtained by applying this mapping to the weighted Sobolev spaces.
\begin{proposition}
    The multiplication operator $e^{-V/2}$ is a unitary transformation from $\sobolev{s}{\real^n}[\operator{L}]$ to $\sobolev{s}{\real^n}[\operator{H}]$,
    where $-\operator{H}$ is the nonnegative selfadjoint operator on $\lp{2}{\real^n}$ defined by
    \[
        -\operator{H} = e^{-V/2} \, \operator{L} \, e^{V/2} = - \Delta + \left(\frac{\abs{\grad{V}}^2}{4} - \frac{\Delta V}{2}\right) =:  -\Delta + W.
    \]
    \begin{proof}
        Since $(-\operator{H})^i = e^{-V/2} \, (-\operator{L})^i \, e^{V/2}$, $\ip{u}{v}_{s,\operator{L}} = \ip{e^{-V/2}u}{e^{-V/2}v}_{s, \operator{H}}$ for any $u,v \in \test{\real^n}$ and any exponent $i \in \nat$,
        from which the result follows by density.
    \end{proof}
\end{proposition}
The space $\sobolev{1}{\real^n}[\operator{H}]$, for $\operator{H}$ defined as above, is of particular relevance to this paper.
It is a simple exercise to show that this space can be equivalently defined by
\[
    \sobolev{1}{\real^n}[\operator{H}] = \left\{ u \in \sobolev{1}{\real^n} : \int_{\real^n} |W| u^2 \, dy < \infty\right\},
\]
and that for $u \in \sobolev{1}{\real^n}[\operator{H}]$,
\[
    \norm{u}[1][\operator{H}]^2 = \norm{u}[0]^2 + \int_{\real^n} \abs{\grad{u}}^2\,dy + \int_{\real^n} Wu^2\,dy.
\]

\section{Hermite Polynomials and Hermite Functions}
\label{sec:hermite}
In this appendix, we recall some results about Hermite polynomials and Hermite functions that are essential for the analysis presented in this paper.
\paragraph{Hermite polynomials}
In one dimension, it is well-known that the polynomials
\begin{equation}
    H_r(s) = \frac{{(-1)}^r}{\sqrt{r!}}
    \exp\left({\frac{s^2}{2}}\right)\frac{\mathrm d ^r}{\mathrm d
    s^r}\left(\exp{\left(\frac{-s^2}{2}\right)}\right) \qquad r = 0, 1, 2, \dots
    \label{Hermite_polynomials_in_1_dimension}
\end{equation}
form a complete orthonormal basis of $\lp{2}{\real}[\gaussian[0][1]]$,
where $\gaussian[0][1]$ is the Gaussian density of mean 0 and variance 1.
These polynomials can be naturally extended to the multidimensional case.
For $\mu \in \real^n$ and a symmetric positive definite matrix $\Sigma \in \real^{n\times n}$, consider the Gaussian density $\gaussian[\mu][\Sigma]$ of mean $\mu$ and covariance matrix $\Sigma$.
Let $D$ and $Q$ be diagonal and orthogonal matrices such that $\Sigma = Q D Q^T$, and note $S = Q D^{1/2}$, such that $\Sigma = SS^T$.
With these definitions, the polynomials defined by
\begin{equation}
    \hermite H_{\alpha}(y; \mu,\Sigma)\,=\, \hermite H^*_{\alpha} (S^{-1}(y-\mu)), \quad \text{ with }
    \alpha \in \nat^n \text{ and } \hermite H^*_\alpha(z)={\prod}_{k\,=\,1}^n H_{\alpha_k}(z_k),
    \label{eq:_Definition_of_modified_Hermite_polynomials}
\end{equation}
form a complete orthonormal basis of $L^2(\real^n, \gaussian[\mu][\Sigma])$.
Note that the Hermite polynomial corresponding to a multi-index $\alpha$ depends on the orthogonal matrix $Q$ chosen.
When $\mu$ and $\Sigma$ are clear from the context, we will sometimes omit them to simplify the notation.

In addition to forming a complete orthonormal basis, the Hermite polynomials defined above are the eigenfunctions of the Ornstein-Ulhenbeck operator
\[
    -\operator{L}_{\mu,\Sigma} = \Sigma^{-1}(y-\mu) \cdot \nabla - \Delta.
\]
The eigenvalue associated to $\hermite{H}_{\alpha}(y;\mu,\Sigma)$ is given by
\begin{equation}
    \label{eq:eigenvalues_of_hermite polynomials}
    \lambda_\alpha = \sum_{i=1}^{n} \alpha_i \lambda_i,
\end{equation}
where $\seq{\lambda}{i}{1}{n}$ are the diagonal elements of $D^{-1}$.
Naturally, the operator $-\operator{L}_{\mu,\Sigma}$ is nonnegative and selfadjoint on $\lp{2}{\real^n}[\gaussian[\mu][\Sigma]]$.

Hermite polynomials have very good approximation properties for smooth functions in $\lp{2}{\real^n}[\gaussian[\mu][\Sigma]]$.
In what follows, we note $\proj{\cdot}{\poly{d}}: L^2(\real^n, \gaussian[\mu][\Sigma]) \rightarrow \poly{d}$ the $L^2(\real^n, \gaussian[\mu][\Sigma])$ projection
operator on the space of polynomials of degree less than or equal to $d$.
\begin{proposition}
    For $u \in \sobolev{s}{\real^n}[\operator{L}_{\mu,\Sigma}]$,
    \[
        \norm{u}[s][\operator{L}_{\mu,\Sigma}]^2 = \sum_{\alpha}(1 + \lambda_\alpha + \lambda_\alpha^2 + \dots + \lambda_\alpha^s) c_\alpha^2, \quad \text{ where } \quad c_\alpha = \ip{u}{\hermite H_\alpha}[\gaussian[\mu][\Sigma]].
    \]
    In addition $u \in \sobolev{s}{\real^n}[\operator{L}_{\mu,\Sigma}]$ if and only if the sum in the right-hand side converges.
    \begin{proof}
        Let $-\operator{L} = \sum_{i=0}^{s} (-\operator{L}_{\mu,\Sigma})^i$ and $\mu_\alpha = 1 + \lambda_\alpha + \lambda_\alpha^2 + \dots + \lambda_\alpha^s$.
        Assume first that $u \in \test{\real^n}$, so $-\operator{L}u \in \test{\real^n}$ also.
        Using the selfadjoint property of $-\operator{L}$, it is clear that $-\operator{L}u = \sum_{\alpha} \mu_\alpha c_\alpha \, \hermite{H}_\alpha$.
        Taking the norm and expanding the functions
        \[
            \norm{u}[s][\operator{L}_{\mu,\Sigma}]^2 = \int_{\real^n} \left(\sum_{\alpha} \mu_\alpha c_\alpha \, \hermite{H}_\alpha\right) \, \left(\sum_{\alpha} c_\alpha \hermite H_\alpha \right) \, \gaussian[\mu][\Sigma] \,dy= \sum_{\alpha} \mu_\alpha c_\alpha^2.
        \]
        We consider now the general case $u \in \sobolev{s}{\real^n}[\operator{L}_{\mu,\Sigma}]$. By definition of $\sobolev{s}{\real^n}[\operator{L}_{\mu,\Sigma}]$ there exists $\seq{u}{n}{1}{\infty} \in \test{\real^n}$ such that $\norm{u-u_n}[s][\operator{L}_{\mu,\Sigma}] \to 0$.
        By the previous equation, this means that $\sum_{\alpha} \mu_\alpha (c_{\alpha,n} - c_{\alpha,m})^2 \to 0$ when $m,n \to \infty$, where $c_{\alpha,k} = \ip{u_k}{\hermite{H}_\alpha}$, which by a completeness argument implies that $\sum_{\alpha} \mu_\alpha c_{\alpha,n}^2 \to \sum_{\alpha} \mu_\alpha c_{\alpha}^2$.

        It remains to show that if the series is convergent, then $u \in \sobolev{s}{\real^n}[\operator{L}_{\mu,\Sigma}]$.
        The main idea, for this part, is to show that the sequence $u_N = \sum_{|\alpha|\leq N} \ip{u}{\hermite H_\alpha}[\gaussian[\mu][\Sigma]] \hermite{H}_\alpha$
        is Cauchy in $\sobolev{s}{\real^n}[\operator{L}_{\mu,\Sigma}]$, which is a routine check.
    \end{proof}
\end{proposition}
Noting that $1 + r + r^2 + \dots + r^s \leq e^{\frac{1}{r}} r^s$ for $r > 0$, and that $\lambda_\alpha \to \infty$ when $|\alpha| \to \infty$, the previous result implies that
\begin{equation}
    \label{eq:weighted_sobolev_norm_in_terms_of_eigenvalues}
    \left(c_0^2  + \sum_{|\alpha| > 0} \lambda_\alpha^s c_\alpha^2\right)  \leq \norm{u}[s][\operator{L}_{\mu,\Sigma}]^2 \leq L\left(c_0^2  + \sum_{|\alpha| > 0} \lambda_\alpha^s c_\alpha^2\right),
\end{equation}
where $L = \max_{\abs{\alpha} > 0} e^{\frac{1}{\lambda_\alpha}}$.
\begin{corollary}[Approximation by polynomials in weighted spaces]
    Let $\Sigma$ be a symmetric positive definite matrix, and suppose that $f \in \sobolev{s}{\real^n}[\operator{L}_{\mu,\Sigma}]$. Then
    \[
        \norm{f- \proj{f}{\poly{d}}}[r][\operator{L}_{\mu,\Sigma}]\, \leq \,C(\Sigma, r, s) \, {d}^{-\frac{(s-r)}{2}}\, \norm{f}[s][\operator{L}_{\mu,\Sigma}],
    \]
    for $r \in \nat$ such that $0 \leq r \leq s$.
    \begin{proof}
        See \cite[Theorem 3.1]{gagelman2012spectral}.
        From~\eqref{eq:weighted_sobolev_norm_in_terms_of_eigenvalues}, we have that
        \[
            \norm{f- \proj{f}{\poly{d}}}[r][\operator{L}_{\mu,\Sigma}]^2 \leq L \sum_{|\alpha| > d} \lambda_\alpha^r c_\alpha^2 \leq L \, M^{r-s}\sum_{\alpha \in \nat^n} \lambda_{\alpha}^{s}c_\alpha^2 \leq L M^{r-s} \norm{f}[s][\operator{L}_{\mu,\Sigma}]^2,
        \]
        with $c_\alpha = \ip{f}{\hermite{H}_\alpha}[0][\operator{L}_{\mu,\Sigma}]$ and $M = \min_{\abs{\alpha}>d} \lambda_\alpha$.
        Since $\lambda_\alpha > C(\Sigma)\, |\alpha|$, the conclusion follows.
    \end{proof}
    \label{proposition:approximation_by_hermite_polynomials}
\end{corollary}

\paragraph{Hermite Functions}
\label{sec:hermite_functions}
Hermite functions can be defined from Hermite polynomials as follows:
\begin{definition}
    \label{definition:hermite_function}
    Given $\mu \in \real^n$, and $\Sigma \in \real^{n\times n}$ positive definite, we define the Hermite functions $\hermitef{h}_{\alpha}(y;\mu,\Sigma)$ by:
    \[
        \hermitef{h}_{\alpha}(y; \mu, \Sigma) = \sqrt{\gaussian[\mu][\Sigma]} \, \hermite{H}_{\alpha}(y; \mu, \Sigma) \quad \text{ for } \alpha \in \nat^n.
    \]
\end{definition}
The Hermite functions form a complete orthonormal basis of $\lp{2}{\real^n}$.
Since they are obtained from the Hermite polynomials by a multiplication with $\sqrt{\gaussian[\mu][\Sigma]}$, we immediately obtain the following:
\begin{proposition}
    Given $\mu \in \real^n$ and $\Sigma \in \real^{n\times n}$ positive definite, the Hermite functions $\hermitef{h}_{\alpha}(y;\mu,\Sigma)$ are the eigenfunctions of the operator:
    \[
        -\operator{H}_{\mu,\Sigma} = {(\gaussian[\mu][\Sigma])}^{\frac{1}{2}} \, (-\operator{L}_{\mu,\Sigma}) \, {(\gaussian[\mu][\Sigma])}^{-\frac{1}{2}} = - \Delta + \left(\frac{(y-\mu)^T \Sigma^{-2} (y-\mu)}{4} - \frac{\trace \Sigma^{-1}}{2}\right),
    \]
    with the same eigenvalues as in~\eqref{eq:eigenvalues_of_hermite polynomials}.
\end{proposition}
Hermite functions inherit the good approximation properties of the Hermite polynomials expressed in \cref{proposition:approximation_by_hermite_polynomials}.
In the following result, $\pi$ refers to the $\lp{2}{\real^n}$ projection operator, so
\begin{equation}
    \proj{f}{\sqrt{\gaussian[\mu][\Sigma]}\poly{d}} = \sum_{\abs{\alpha} \leq d} \ip{f}{h_\alpha(\cdot; \mu, \Sigma)} \, h_\alpha(\cdot; \mu, \Sigma).
\end{equation}
\begin{corollary}[Approximation by Hermite functions in flat space]
    Let $\mu \in \real^n$ and $\Sigma \in \real^{n \times n}$ be a symmetric positive definite matrix, and suppose that $f \in H^s(\real^n,\,\operator{H}_{\mu,\Sigma})$. Then
    \[
        \norm{f- \proj{f}{\sqrt{\gaussian[\mu][\Sigma]}\poly{d}}}[r][\operator{H}_{\mu,\Sigma}]\, \leq \,C(\Sigma, r, s) \, {d}^{-\frac{(s-r)}{2}}\, \norm{f}[s][\operator{H}_{\mu,\Sigma}],
    \]
    for any $r \in \nat$ such that $0 \leq r \leq s$.
\label{corollary:approximation_by_hermite_functions}
\end{corollary}

\nocite*
\bibliographystyle{abbrv}
\bibliography{references}

\end{document}